\documentclass[12pt]{article}
\usepackage{geometry}                
\usepackage{graphicx,color,mathtools}
\usepackage{amssymb,amsmath,amsthm,mathrsfs}
\usepackage[all,cmtip]{xy}
\usepackage{epstopdf, comment, url}
\usepackage{bm} 
\usepackage{enumitem}

\newtheorem{theorem}{Theorem}[section]
\newtheorem{corollary}[theorem]{Corollary}
\newtheorem{lemma}[theorem]{Lemma}

\theoremstyle{remark}
\newtheorem*{remark}{Remark}

\numberwithin{equation}{section}

\DeclareMathOperator{\crit}{crit}

\DeclareMathOperator{\diam}{diam}

\DeclareMathOperator{\aut}{Aut}

\DeclareMathOperator{\rad}{rad}

\DeclareMathOperator{\re}{Re}

\DeclareMathOperator{\BMO}{BMO}

\DeclareMathOperator{\GP}{AI}

\title{Analytic mappings of the unit disk which almost preserve hyperbolic area}
\author{Oleg Ivrii and Artur Nicolau  }
\date{September 20, 2024}

\begin{document}
 
\maketitle

\begin{abstract}
In this paper, we study analytic self-maps of the unit disk which distort hyperbolic areas of large hyperbolic disks by a bounded amount. We give a number of characterizations involving angular derivatives, Lipschitz extensions, Möbius distortion, the distribution of critical points and Aleksandrov-Clark measures. 
We also examine the Lyapunov exponents of their Aleksandrov-Clark measures.
\end{abstract}

\section{Introduction}

Let $\mathbb{D}$ be the open unit disk in the complex plane and $\lambda (z) |dz| = 2 |dz| / (1-|z|^2)$ be the hyperbolic metric in $\mathbb{D}$. We denote the  hyperbolic distance between the points $z, w \in \mathbb{D}$ by $d_h (z, w)$ and the hyperbolic area of a measurable set $E \subset \mathbb{D}$ by
$$
A_h (E) = \int_E \frac{4\, dA(z)}{(1-|z|^2)^2 } . 
$$
Let $F: \mathbb{D} \rightarrow \mathbb{D}$ be an analytic self-mapping of the unit disk. The pullback of the hyperbolic metric under $F$ is $\lambda_{F} (z) |dz| = 2 |dz| \cdot |F'(z)| / (1-|F(z)|^2)$, $z \in \mathbb{D}$. The Schwarz lemma says that any analytic self-mapping $F$ of the unit disk contracts the hyperbolic metric, that is, $\lambda_F (z) \leq \lambda (z)$ for any $z \in \mathbb{D}$. Moreover, equality at a single point implies equality at every point in the unit disk and that $F$ is an automorphism of $\mathbb{D}$. Define the {\em Möbius distortion} of $F$ as 
$$
\mu(z) \ = \ 1 -  \frac{\lambda_F (z)}{\lambda (z)}   \ = \ 1 - \frac{(1-|z|^2) |F'(z)|}{1-|F(z)|^2}, \qquad z \in \mathbb{D}.
$$
Since $\mu$ vanishes identically if and only if $F$ is an automorphism of $\mathbb{D}$, the Möbius distortion of $F$ measures how much $F(z)$ deviates from an automorphism of $\mathbb{D}$ near $z \in \mathbb{D}$.

Let $m = |dz|/2\pi$ be the normalized Lebesgue measure on the unit circle $\partial \mathbb{D}$. Inner functions are analytic self-mappings $F$ of the unit disk such that their radial limits $\lim_{r \to 1} F(r \xi)$ have modulus one at $m$ almost every point  $\xi \in \partial \mathbb{D}$. An inner function $F$ has {\em finite entropy} if its derivative $F'$ belongs to the Nevanlinna class $\mathcal N$ of analytic functions $g$ in $\mathbb{D}$ such that
$$
\limsup_{r \to 1} \int_{\partial \mathbb{D}} {\log}^+ |g(r \xi)| dm(\xi) < \infty. 
$$
Inner functions with finite entropy have been extensively studied  \cite{craizer, inner, stable, IK22, IN22, inner-tdf, laminations} 
 and play an important role in our discussion.  An analytic self-mapping $F$ of the unit disk is said to have a {\em finite
angular derivative} (in the sense of Carath\'eodory) at a point $\xi \in \partial \mathbb{D}$ if $\lim_{r \to 1} F(r\xi )$ belongs to the unit circle and $ \lim_{r \to 1} F' (r \xi) $ exists. In this case, we write $|F' (\xi)| = \lim_{r \to 1} |F' (r \xi)|$. If $F$ does not have a finite angular derivative at $\xi$, it is customary to set $|F'(\xi)| := \infty$. 

The hyperbolic disk centered at $z \in \mathbb{D}$ of hyperbolic radius $R>0$ is denoted by $B_h (z, R)$. We say that an analytic self-mapping $F$ of the unit disk {\em almost preserves hyperbolic area} (APHA) if there exists a constant $c>0$ such that
\begin{equation}
\label{eq:area-condition}
A_h (F(B_{h}(z, R)) \ge c  \, A_h (B_{h}(z, R)), \quad z \in \mathbb{D}, \quad R > 1.
\end{equation}
In view of the Schwarz lemma, \eqref{eq:area-condition} says that the hyperbolic area of the image of any hyperbolic disk is almost as large as possible. In this paper, we present a number of descriptions of APHA mappings in terms of angular derivatives, Lipschitz extensions, Möbius distortion and the  distribution of critical points. Our main motivation is the paper \cite{GP} of J.~Garnett and M.~Papadimitriakis which concerns almost isometries. An analytic self-mapping $F$ of the unit disk is called an {\em almost isometry} if there exists a constant $c>0$ such that ${\diam}_h F(B_{h}(z,R)) \ge 2R - c,$ for any $z \in \mathbb{D}$ and $R>0$. Garnett and Papadimitriakis observed that almost isometries are necessarily Blaschke products and found an elegant description in terms of angular derivatives. 

 If $C$ is the critical set of some Blaschke product, then it is the critical set of a particular Blaschke product $F = F_C$ called the {\em maximal Blaschke product} which maximizes $\lambda_F(z)$
out of all Blaschke products with critical set $C$. The maximal Blaschke product $F_C$  is uniquely determined up to post-composition with M\"obius transformations. For general properties of maximal products, we refer the reader to \cite{kraus, KR-survey, KR-maximal}.

We use the following notation: For a point $z \in \mathbb{D}$, we write $\omega (z, E, \mathbb{D})$ for the harmonic measure seen from $z$ of a measurable set $E \subset \partial \mathbb{D} $ in the unit disk, that is,
$$
\omega (z, E, \mathbb{D}) = \int_E \frac{1-|z|^2}{|\xi - z|^2} \, dm(\xi). 
$$
The classical L\"owner Lemma says that $\omega (F(z), E, \mathbb{D}) = \omega (z, F^{-1} (E), \mathbb{D})$ for any point $z \in \mathbb{D}$,   measurable set $E \subset \partial \mathbb{D}$ and inner function $F$.

If $z \in \mathbb{D} \setminus \{0\}$, we write $I_z$ for the arc of the unit circle centered at $z/|z|$ with $m(I_z) =  1-|z|$. 
More generally, for $K > 0$, we write $K I_z$ for the arc centered at $z/|z|$ with  $m(I_z) = K(1-|z|)$. Conversely, given an arc $I \subset \partial \mathbb{D}$ centered at a point $\xi \in \partial \mathbb{D}$, we write $z_I = (1- m(I)) \xi$. 

For an arc  $I \subset \partial \mathbb{D}$, we write
$$
Q_I = \{ z \in \mathbb{D} : z/|z| \in I, \, 1 - m(I) < |z| < 1 \}
$$
for the Carleson square with base $I$. Given a Carleson square $Q=Q_I$, we denote its ``side length'' by $\ell (Q) = m(I)$.

 Our first main result provides several equivalent characterizations of APHA mappings:

\begin{theorem}
\label{main-thm}
Let $F: \mathbb{D} \to \mathbb{D}$ be an analytic self-map of the unit disk with Möbius distortion $\mu$. The following statements are equivalent:

\begin{enumerate}
\item[{\em (1)}] $F$ is an APHA mapping, that is, there exists a constant $c > 0$ such that 
$$A_h (F(B_{h}(z,R)) > c \, A_h (B_{h}(z,R)), $$ 
for any $z \in \mathbb{D}$ and any $R > 1$. 

\item [{\em (2)}] There exist constants $C, \delta > 0$ so that for any point $z \in \mathbb{D}$ one has 
$$
\omega \biggl( z, \left\{ \xi \in \partial \mathbb{D} : |F' (\xi)| < C \cdot \frac{1-|F(z)|}{1-|z|}\right\}, \mathbb{D} \biggr) \ge \delta.
$$

\item[{\em (3)}] The angular derivative $|F' (\xi)|$ is finite at almost every point $\xi \in \partial \mathbb{D}$, the function   $\log |F'|$ is integrable on the unit circle and there exists a constant $C>0$ such that for any arc $I \subset \partial \mathbb{D}$, one has
\begin{equation}
    \label{mean}
 \log \frac{1-|F(z_I)|}{1-|z_I|} - C \, \leq \,  \frac{1}{m(I)} \int_I \log |F'| dm  \, \leq \,  \log \frac{1-|F(z_I)|}{1-|z_I|} + C .
\end{equation}

\item [{\em (4)}] $F$ is an inner function of finite entropy and there exists a constant $C>0$ so that the outer part $O_{F'}$ of $F'$ satisfies 
\begin{equation}
    \label{outer}
C^{-1} \cdot \frac{1-|F(z)|}{1-|z|} \, \leq \, |O_{F'}(z)| \, \leq \, C \cdot \frac{1-|F(z)|}{1-|z|}, \qquad z \in \mathbb{D}.   
\end{equation}

\item[{\em (5)}] 
$F$ is an inner function such that for any $\varepsilon > 0$, there exists a constant $C(\varepsilon ) >0$ so that for any arc $I \subset \partial \mathbb{D}$, one can find pairwise disjoint subarcs $J_k \subset I$ with $\sum m(J_k) \le \varepsilon \, m(I)$ and 
\begin{equation}
\label{lipschitz}
|F'(z)| \le C(\varepsilon) \cdot  \frac{1-|F(z_I)|}{1-|z_I|}, \quad z \in  Q_I \setminus \bigcup_k Q_{J_k} .
\end{equation}

\item[{\em (6)}] There exists a constant $C>0$ so that for any arc $I \subset \partial \mathbb{D}$ we have
\begin{equation}
\label{estimatedistortion}
\int_{Q_I} \mu(z) \,\frac{dA(z)}{1-|z|} \leq C m(I).
\end{equation}

\item[{\em (7)}] $F$ is a maximal Blaschke product and there exists a constant $C>0$ so that for any arc $I \subset \partial \mathbb{D}$, we have
$$
\sum_{c \in  Q_I:\, F'(c)=0} (1 - |c|) \leq C m(I).
$$
\end{enumerate}
\end{theorem}

\subsection*{Remarks}
 
\paragraph*{1.} J. Garnett and M. Papadimitrakis proved in \cite{GP} that an analytic self-mapping $F$ of the unit disk is an almost isometry if and only if there exists a constant $C>0$ so that for any arc $I$ of the unit circle there is a point $\xi \in I$ with
$$
|F'(\xi)| < C \cdot \frac{1-|F(z_I)|}{1-|z_I|}.
$$
One can view Statement (2) as an analogue of this result where one replaces {\em existence} with {\em abundance}.    

 \vspace{-4pt}

\paragraph*{2.} By Corollary \ref{a-lower-bound-angular-derviative} in Section 3, there exists a constant $C>0$ such that for any analytic self-mapping $F$ of the disk, we have 
      $$
  \log \frac{1-|F(z_I)|}{1-|z_I|} \le \frac{1}{m(I)} \int_I \log |F'| dm + C.
$$
In other words, the lower bound in \eqref{mean} always holds.

\paragraph*{3.}  Let $\BMO(\partial \mathbb{D})$ be the space of integrable functions $h$ on $\partial \mathbb{D}$ such that
$$
\sup \frac{1}{m(I)} \int_I \, \biggr| h(\xi) - \frac{1}{m(I)} \int_I h \biggr| \, dm(\xi) < \infty,
$$
where the supremum is taken over all arcs $I \subset \partial \mathbb{D}$. We will see in Section \ref{sec:equivalence-2345} that Statement (3) can be replaced by a seemingly stronger one:

$(3')$ The angular derivative $|F' (\xi)|$ is finite at almost any point $\xi \in \partial \mathbb{D}$, the function   $\log |F'|$ is in $\BMO(\partial \mathbb{D})$ and for any arc $I \subset \partial \mathbb{D}$, one has
\begin{equation*}
\frac{1}{m(I)} \int_I \log |F'| dm = \log \frac{1-|F(z_I)|}{1-|z_I|} + O(1),
\end{equation*}
where $O(1)$ denotes a quantity which is bounded by a constant independent of the arc $I$. 

However, not every inner function $F$ with $\log|F'| \in \BMO(\partial \mathbb{D})$ satisfies the hypotheses of Theorem \ref{main-thm}. For instance, the singular inner function $F(z) = \exp \bigl ( \frac{z+1}{z-1} \bigr )$ has  $\log|F'(\xi)| = 2 \log |1-\xi|^{-2} \in \BMO(\partial \mathbb{D})$ but is not an APHA mapping since it is not a Blaschke product, see Lemma \ref{island}.

 \vspace{-4pt}

\paragraph*{4.} 
The {\em Lyapunov exponent} of a Borel probability measure $\sigma$ on $\partial \mathbb{D}$ is defined as 
$$
\chi (\sigma , F ) = \int_{\partial \mathbb{D}} \log |F'| d \sigma.  
$$
Typically, the Lyapunov exponent is studied for invariant measures, in which case, it measures the average rate of expansion under forward iteration, but the above definition makes sense for arbitrary measures.

Let $F$ be an inner function of finite entropy. For $p \in \mathbb{D}$, consider the mapping $F_p = \tau_p \circ F$, where $\tau_p$ is an automorphism of $\mathbb{D}$ with $\tau_p (F(p))= p$, so that $p$ is a fixed point of $F_p : \mathbb{D} \rightarrow \mathbb{D}$. By  L\"owner's lemma, $\omega_p (\xi) = \omega(p, \xi , \mathbb{D})$ is $F_p$-invariant, that is,
$\omega_p (F_p^{-1} (E)) = \omega_p (E)$ for any measurable set $E \subset \partial \mathbb{D}$. We define
$$
\chi (\omega_p , F_p) = \int_{\partial \mathbb{D}} \log |F'_p (\xi)| d\omega_p(\xi). 
$$
By the chain rule, 
$$
\chi (\omega_p , F_p) = \int_{\partial \mathbb{D}} \log |F' (\xi)| d\omega_p(\xi) + \int_{\partial \mathbb{D}} \log |\tau'_p  (F(\xi))| d\omega_p(\xi).
$$
As the second integral is $\log |\tau '_p (F(p))| = \log ((1-|p|^2) / (1- |F(p)|^2 )$, we have
$$
\chi (\omega_p , F_p)  = \int_{\partial \mathbb{D}} \log |F' (\xi)| d\omega_p(\xi) + \log \frac{1-|p|^2}{1-|F(p)|^2}. 
$$
Statement (4) in Theorem \ref{main-thm} now gives the following description of APHA maps:

\begin{corollary}
    \label{conformalinvariant}
Let $F$ be an inner function of finite entropy. Then $F$ is an APHA mapping if and only if
$$
\sup_{p \in \mathbb{D}} \chi (\omega_p , F_p) < \infty . 
$$
\end{corollary}
    
A similar computation involving L\"owner's lemma shows that if $F, G$ are inner functions of finite entropy, then
$\chi(\omega_p, F \circ G) = \chi(\omega_p,G) + \chi(\omega_{G(p)}, F).$
Together with Corollary \ref{conformalinvariant}, this implies that APHA mappings are closed under composition.

 \vspace{-4pt}

\paragraph*{5.} For an inner function $F$ of finite entropy, let $ F' = I_{F'} O_{F'}$ be its inner-outer factorization. The estimate \eqref{outer} is related to a reverse Schwarz-Pick inequality due to K.~Dyakonov \cite[Corollary 2.1]{dyakonov2}, which says that if $F$ is an inner function of finite entropy then 
\begin{equation}
\label{eq:dyakonov}
\frac{1-|F(z)|^2}{1-|z|^2} \le      | O_{F'} (z)|, \qquad z \in \mathbb{D}.
\end{equation}
In other words, Statement (4) says that APHA mappings are precisely the ones for which Dyakonov's inequality is sharp up to a bounded factor.

 \vspace{-4pt}

\paragraph*{6.}  The estimate on $|F' (z)|$ in Condition (5) says that $F$ is a Lipschitz function on $Q_I \setminus \bigcup Q_{J_k}$ with Lipschitz constant bounded by $$ C(\varepsilon) \cdot \frac{1-|F(z_I)|}{1-|z_I|}.$$ 
In particular, $\diam F(Q_I \setminus \bigcup Q_{J_k}) \leq C(\varepsilon) \cdot (1-|F(z_I)|)$.

 \vspace{-4pt}

\paragraph*{7.}  A discrete set of points $Z$ in the unit disk is called a {\em Blaschke sequence} if $$\sum_{z \in Z} (1- |z|) < \infty.$$ A Blaschke product $B$ is called a {\em Carleson-Newman Blaschke product} if
      there exists a constant $C>0$ such that 
      $$
      \sum_{z \in Q_I: \, B(z) = 0} (1-|z|) \leq C m(I),
      $$
      for any arc $I \subset \partial \mathbb{D}$. This condition may be understood as the scale invariant version of the Blaschke condition.
      
   Similarly, (\ref{estimatedistortion}) can be understood as the scale invariant version of the estimate
      $$
      \int_{\mathbb{D}} \mu(z) \, \frac{dA(z)}{1-|z|} \lesssim \int_{\partial \mathbb{D}} \log |F' (\xi)| dm (\xi),
      $$
      which holds for any centered inner function $F$ of finite entropy, see Lemma \ref{finite-entropy-characterization}.

 \vspace{-4pt}

\paragraph*{8.} 
From Jensen's formula, it is not difficult to see that if $F$ is an inner function of finite entropy, then its critical set is a Blaschke sequence. Conversely, D.~Kraus \cite[Theorem 4.4]{kraus} noticed that any Blaschke sequence $C$ arises as the critical set of a Blaschke product and the maximal Blaschke product $F_C$ has finite entropy. In other words, one has a bijection between Blaschke sequences and maximal Blaschke products with derivative in Nevanlinna class. Statement (7) says that there is a bijection between APHA mappings and sequences of points in the unit disk satisfying the Carleson condition.

\subsection{Aleksandrov-Clark measures}

Given an analytic mapping $F$ from the unit disk to itself (not necessarily inner) and a point $\alpha \in \partial \mathbb{D}$,
    the function $(\alpha + F)/(\alpha - F)$ has positive real part
    and hence there exists a positive measure $\sigma_\alpha = \sigma_\alpha (F)$ on the unit circle
    and a constant $C_{\alpha} \in \mathbb{R}$ such that
    \begin{equation}
        \label{eq:ACMeasureHerglotz}
        \frac{\alpha + F(z)}{\alpha - F(z)} = \int_{\partial \mathbb{D}} \frac{\xi + z}{\xi - z}\, d\sigma_\alpha(\xi) + iC_{\alpha}, \qquad z \in  \mathbb{D}. 
    \end{equation}
    The measures $\{\sigma_\alpha\colon \alpha \in \partial \mathbb{D}\}$ are called the {\em Aleksandrov-Clark} measures of the function $F$.
    These measures have been introduced by D.~Clark in relation with operator theory
    and were throughly investigated by A.~B.~Aleksandrov who recognized their importance in function theory.
We highlight several elementary properties of Aleksandrov-Clark measures that will be used throughout this work:

\begin{itemize}
  \item If $F(0)=0$, then $\{\sigma_\alpha\colon \alpha \in \partial \mathbb{D}\}$ are probability measures.
  \item
 $F$ is an inner function if and only if $\sigma_\alpha$ is a singular measure for some (and hence, all) $\alpha \in \partial \mathbb{D}$.
 \item   If $F$ is an inner function, then the mass of $\sigma_\alpha$ is carried by the set $F^{-1}(\{\alpha\}) \subset \partial \mathbb{D}.$
 \item The measure $\sigma_{\alpha}$ varies continuously in $\alpha$ in the weak topology of measures on the unit circle.
 \item $F$ has a finite angular derivative at a point $\beta \in \partial \mathbb{D}$ if and only if $\sigma_{F(\beta)}$ has an atom at $\beta$. In this case, $\sigma_{F(\beta)}(\{ \beta \}) = 1/|F'(\beta)|$.
\end{itemize}

    We refer the reader to the surveys  \cite{PolSar, Saks}
    as well as \cite[Chapter~IX]{cima} for further properties of Aleksandrov-Clark measures and a wide range of applications. 
    
The following theorem describes Aleksandrov-Clark measures of APHA mappings. Given a positive Borel measure $\sigma$ on the unit circle, consider the function
\begin{equation}
    \label{Hdefinition}
H[\sigma](z) = \int_{\partial \mathbb{D}} \frac{d\sigma(\xi)}{|\xi - z|^2}, \qquad z \in \overline{\mathbb{D}}.
\end{equation}

\begin{theorem}
\label{ACofAPHAM}
Let $\sigma$ be a positive Borel measure on $\partial \mathbb{D}$ and $H[\sigma]$ be the function defined in \eqref{Hdefinition}. The following statements are equivalent:

\item[{\em (1)}] The measure $\sigma$ is an Aleksandrov-Clark measure of some APHA mapping and some $\alpha \in \partial \mathbb{D}$. 

 \item[{\em (2)}] The function $\log H[\sigma] \in \BMO (\partial \mathbb{D}) $ and there exists a constant $C>0$ so that 
$$
\log H[\sigma](z_I) - C \, \leq \, \frac{1}{m(I)} \int_I \log H[\sigma](\xi) dm (\xi) \, \leq \, \log H[\sigma](z_I) + C,
$$
for any arc $I \subset \partial \mathbb{D}$.
\end{theorem}

The proof of Theorem \ref{ACofAPHAM} relies on Theorem \ref{main-thm}. We will see that the lower bound in Condition (2) holds for any positive Borel measure $\sigma$, so only the upper bound is essential. 

\subsection{Lyapunov exponents}

Finally, we study the Lyapunov exponents of Aleksandrov-Clark measures of APHA mappings. Let $F$ be an analytic self-mapping of the unit disk and $\{\sigma_\alpha : \alpha \in \partial \mathbb{D} \}$ be the family of its Aleksandrov-Clark measures. It is well-known that $F$ has angular derivatives at $m$ almost every point on the unit circle if and only if for $m$ a.e.~$\alpha \in \partial \mathbb{D}$, the measure $\sigma_\alpha$ is discrete, e.g.~see \cite[Theorem 9.6.1]{cima} or \cite[Theorem 3.1]{inner-tdf}. In this case, 
$$
\sigma_\alpha = \sum_{F(\beta)=\alpha} |F'(\beta)|^{-1} \delta_\beta, \qquad \alpha \in \partial \mathbb{D}, 
$$
and the Lyapunov exponent $\chi(\sigma_\alpha, F)$ of $\sigma_\alpha$ is given by 
$$
\chi(\sigma_\alpha, F) \, = \, \int_{\partial \mathbb{D}} \log |F'(\xi)| d \sigma_\alpha (\xi) \, = \, \sum_{F(\beta)=\alpha} |F'(\beta)|^{-1} \log |F'(\beta)|,  \qquad \alpha \in \partial \mathbb{D}. 
$$
We will show in Lemma \ref{nevanlinna-ac} that if $F$ is an inner function, then $F$ has finite entropy if and only if for $m$ a.e.~$\alpha \in \partial \mathbb{D}$, the measure $\sigma_\alpha$ is discrete and 
$$
\int_{\partial \mathbb{D}} \chi(\sigma_\alpha, F) dm (\alpha) < \infty . 
$$
The Lyapunov exponents of the Aleksandrov-Clark measures of APHA mappings satisfy a stronger condition:

\begin{theorem}
\label{ac-measures-for-bmo-inner-functions}
Let $F$ be an APHA mapping and $\{\sigma_\alpha : \alpha \in \partial \mathbb{D} \}$ be the collection of its Aleksandrov-Clark measures. Then, $\chi(\sigma_\alpha, F)$ agrees with a $C^\infty$ function a.e.~on the unit circle.
\end{theorem}

We will see in Lemma \ref{lsc-lyapunov} that the Lyapunov exponent $\chi(\sigma_\alpha, F)$ defines a lower semicontinuous function on the unit circle. Together with the above theorem, this implies that the measure $\sigma_\alpha$ is discrete and $\chi(\sigma_\alpha, F)$ is finite for every $\alpha \in \partial \mathbb{D}$.

The main idea in the proof of Theorem \ref{ac-measures-for-bmo-inner-functions} is that the Lyapunov exponent $\chi(\sigma_\alpha, F)$ shows up naturally when studying a certain family of weighted composition operators.

\subsection*{Remarks}
 
\paragraph*{1.} By using \cite[Lemma 8.3]{inner-tdf} in place of Lemma \ref{sigma-half-plane} below, one can show that the conclusion of Theorem \ref{ac-measures-for-bmo-inner-functions} also holds for one component inner functions. While the two classes have a number of similarities,
it is not difficult to come up with examples showing that neither class contains the other.


 \vspace{-4pt}

\paragraph*{2.} Suppose that $F$ satisfies the conditions of Theorem \ref{main-thm}. As $\log |F'| \in \BMO(\partial \mathbb{D})$, the John-Nirenberg Theorem implies that $F' \in H^p$ for some $p > 0$. An argument similar to the one in the proof of Lemma \ref{nevanlinna-ac} below shows that
\begin{equation*}
 \chi_p(\sigma_\alpha, F) \, = \, \sum_{F(\beta)=\alpha} |F'(\beta)|^{-1+p} \, < \, \infty,
\end{equation*}
for $m$ a.e.~$\alpha \in \partial \mathbb{D}$. However, we are unable to give a uniform bound on $\chi_p$, valid for every $\alpha \in \partial \mathbb{D}$.

 \vspace{-4pt}

\paragraph*{3.} Using Condition 3 in Theorem \ref{main-thm}, it is not difficult to check that an inner function $F$ is an APHA mapping if and only if $zF(z)$ is an APHA mapping. Therefore, if $F$ is an APHA mapping with an attracting fixed point at the origin, then the Aleksandrov-Clark measure of $\alpha=1$ for $F(z)/z$ is discrete. In particular, if $F$ is not a finite Blaschke product, then it has infinitely many boundary-repelling fixed points $\xi$ on the unit circle with finite multipliers $F'(\xi) = \lim_{r \to 1} F'(r\xi) > 1$. 
Applying these considerations to the iterates of $F$ shows that $F$ has infinitely many boundary-repelling periodic orbits  of any order  on the unit circle with finite multipliers.

\bigskip

The paper is organized as follows. Section 2 contains some preliminary observations. In particular, we show that APHA maps are inner functions and discuss the condition $R>1$ in definition \eqref{eq:area-condition}. The equivalence of Statements (1) and (2) in Theorem \ref{main-thm} is proved in Section 3. Section 4 is devoted to the proof of the equivalence of Statements (2), (3), (4) and (5) in Theorem \ref{main-thm} while the equivalence of Statements (3), (4), (6) and (7) is discussed in Section 5. The proof of Theorem \ref{ACofAPHAM} is given in Section 6, while the proof of Theorem \ref{ac-measures-for-bmo-inner-functions} is presented in Section 7.

      
\section{Preliminary observations}

In this section, we make a number of simple observations about APHA mappings and related concepts.

\begin{lemma}
\label{island}
Let $F$ be an APHA mapping. Then $F$ is an inner function. In fact, $F$ is an indestructible Blaschke product, that is, $\tau \circ F$ is a Blaschke product for any M\"obius transformation $\tau \in \aut(\mathbb{D})$.
\end{lemma}

\begin{proof}
By \cite[Lemma 1.8]{mashreghi}, it is enough to show that if $F$ has a non-tangential limit at a point $\xi \in \partial \mathbb{D}$, then it cannot lie in the unit disk.
On the contrary, if the non-tangential limit at $\xi \in \partial \mathbb{D}$ is $a \in \mathbb{D}$, then
 there exist $r_n \to 1$ and $R_n \to \infty$ such that
$ F(B_{h}(r_n\xi, R_n)) \subset B_h(a,1)$, $n \in \mathbb{N}$, contradicting
  (\ref{eq:area-condition}).
 \end{proof}

The following lemma says that finite Blaschke products are APHA mappings:
    
   \begin{lemma}
   	Suppose $F$ is a finite Blaschke product. Then there exists a constant $c > 0$ so that
	 $F(B_{h}(z,R)) \supset B_{h}(F(z),R-c)$ for any $z \in \mathbb{D}$ and $R > 0$.
   \end{lemma}
 
   \begin{proof}
   By composing with a M\"obius transformation, we may assume that $F$ is centered, i.e.~$F(0) = 0$.
   We will use \cite[Theorem 10.11]{mcmullen-rtree} which says that there exists a constant $R_0 > 0$, depending on the degree of $F$, such that if the hyperbolic distance from a geodesic segment $[a, b]$ to the critical values of $F$ is at least $R_0$, then
 $$
   d_h(F^{-1}(a), F^{-1}(b)) = d_h(a,b) + O(1),
   $$
   where $F^{-1}$ is a branch of the inverse map that is continuous along $[a,b]$.
Choose $R_1 > 0$ sufficiently large so that the disk $B_{h}(0, R_1)$ contains the critical values of $F$.

Let $w = F(z)$. Given a point $w' \in \mathbb{D}$, we need to produce a pre-image $z'$ with $d_h(z, z') \le d(w, w') + O(1)$. For this purpose, join $w$ and $w'$ by a hyperbolic geodesic $[w, w']$.
  We first examine two special cases:
  
  \begin{enumerate} 
 \item If $[w, w']$ is contained in $B_{h}(0, R_0 + R_1)$, then for any pre-image $z'$ of $w'$, we have 
$d_h(z, z') \le \diam_h F^{-1}(B_{h}(0, R_0 + R_1))$.

 \item If $[w,w']$ does not pass through $B_{h}(0, R_0 + R_1)$ then one may define $z'$ as the pre-image of $w'$ obtained by analytically continuing
 $F^{-1}(w) = z$ along the geodesic $[w,w']$. The desired estimate then follows from McMullen's theorem mentioned above.
 \end{enumerate}
 
 In the general case, the geodesic $[w, w']$ may intersect $\partial B(0, R_0 + R_1)$ in at most two points, which divide $[w, w']$ in at most 3 pieces that satisfy the assumptions of one of the two special cases.
  \end{proof}
  
Next we show that in the definition (\ref{eq:area-condition}) of APHA mappings one can replace the condition ``for any $R > 1$'' with ``for any $R$ sufficiently large.''
The proof is based on the following two lemmas:

   \begin{lemma}
   \label{lambda-lemma1} Let $F$ be an analytic self-mapping of the unit disk. For any $\delta > 0$, there exists a constant $\rho = \rho (\delta) > 0$ so that if $(\lambda_F/\lambda)(z) > \delta$ then $F$ is injective on $B_{h}(z, \rho)$ and $(\lambda_F/\lambda)(w) > \delta/2$ for $w \in B_{h}(z, \rho)$. 
   \end{lemma}
   
   \begin{proof}
   Composing with M\"obius transformations, we may assume that $z = 0$ and $F(0) = 0$. The assumption then says that
   $|F'(0)| > \delta$. A normal families argument shows that there exists a $0 < \rho < 1/10$ so that
   $$|F(w)| < \frac{\delta}{10}, \qquad |F'(w) - F'(0)| < \frac{1}{10} \, |F'(0)|, \qquad \text{for } w \in B_{h}(0, \rho),$$ from which the assertions of the lemma follow easily. 
   \end{proof}
   
A similar argument involving normalizing $F(0) = 0$ and normal families shows:

      \begin{lemma}
         \label{lambda-lemma2}
 Let $F$ be an analytic self-mapping of the unit disk. For any $R > 0$ and $\varepsilon > 0$, there exists a constant $\delta > 0$ so that if 
  $$(\lambda_F/\lambda)(w) < \delta, \qquad w \in B_{h}(z,1/2), $$ then 
$(\lambda_F/\lambda)(w) < \varepsilon$ for any $w \in B_{h}(z,R)$.
   \end{lemma}
   
With the help of the two lemmas above, we show that if (\ref{eq:area-condition}) fails for $R = 1$, then it also fails for any given $R > 1$. Indeed, by
   Lemma \ref{lambda-lemma1}, if the hyperbolic area of $F(B_{h}(z,1))$ is small, then $ \sup \bigl \{(\lambda_F/\lambda)(w) : w \in B_h (z,1/2) \bigr \} $ is small.  
   Lemma \ref{lambda-lemma2} then implies that $(\lambda_F/\lambda)(w)$ is small on $B_{h}(z,R)$, and so
   $F(B_{h}(z,R))$ also has small hyperbolic area.

To conclude this section, we discuss two notions which on the surface resemble the APHA condition but turn out to be quite different.
We first show that automorphisms of the disk are the only analytic self-mappings of the disk for which the estimate in condition \eqref{eq:area-condition} holds for any $R>0$.

 \begin{lemma}
         \label{lambda-lemmaenyR}
  Let $F$ be an analytic self-mapping of the unit disk. Assume that there exists a constant $c>0$ such that 
  $ A_h (F(B_{h}(z, R)) \ge c\, A_h (B_{h}(z, R)),$ for any $  z \in \mathbb{D}$ and any  $R > 0$. Then $F$ is a M\"obius transformation.
   \end{lemma}

\begin{proof}
The assumption implies that $(\lambda_F/\lambda)(z) \gtrsim c$ is bounded below on the unit disk by a positive constant.
By \cite[Theorem 1.1]{KRR}, $F$ extends analytically past the unit circle and so is a finite Blaschke product.
Since $F$ cannot have critical points, it must be a M\"obius transformation.
\end{proof}

We now give an example of an inner function which is not an APHA mapping but almost preserves hyperbolic area if the image is counted with multiplicity, i.e.~
\begin{equation}
\label{eq:area-condition-with-multiplicity}
\int_{B_{h}(z,R)} \frac{4 |F'(z)|^2}{(1-|F(z)|^2)^2} \, dA(z) \ge c  \, A_h (B_{h}(z, R)), \quad z \in \mathbb{D}, \quad R > 1,
\end{equation}
for some $c > 0$.

Let $\Gamma$ be any co-compact Fuchsian group acting on the unit disk. It follows from the solution of the Schwarz-Picard problem given in \cite[Theorem 21.1]{heins} that one can find a maximal Blaschke product $F$ whose critical set is an orbit of $\Gamma$. Since the critical set determines a maximal Blaschke product up to post-composition with an element of $\aut(\mathbb{D})$, $F$ satisfies the automorphy relation
$$F(\gamma(z)) = m_\gamma(F(z)),$$ for some character $\chi: \Gamma \to \aut(\mathbb{D})$. In particular, the hyperbolic area of the image of any disk of hyperbolic radius 1 is bounded below. A fortiori, $F$ satisfies (\ref{eq:area-condition-with-multiplicity}).
However, since the critical points of $F$ do not satisfy the Blaschke condition, $F$ does not have finite entropy and is therefore not an APHA mapping, see Condition (4) of Theorem \ref{main-thm}.

\section{Almost Isometric Rays}

\subsection{Almost isometric and Good Geodesic Rays}

Let $F$ be an analytic self-mapping of the unit disk. For a geodesic ray $\gamma = [z, \xi)$ emanating from $z \in \mathbb{D}$ and ending at $\xi \in \partial \mathbb{D}$, the {\em lag function}
\begin{equation}
    \label{lag}
L(w) = d_h(z, w) - d_h(F(z), F(w)), \qquad w \in \gamma, 
\end{equation}
measures the contraction of the hyperbolic metric along $\gamma$.
From the Schwarz lemma and the triangle inequality, it is clear that $L(w)$ is non-negative and increases as $d_h(z, w)$ increases.
We define
$$
L(\xi) := \sup_{w \in \gamma} \bigl \{ d_h(z, w) - d_h(F(z), F(w)) \bigr \},
$$
which may be infinite.

We say that $\gamma$ is a $C$-{\em almost isometric} ray if the lag function of $\gamma$ is bounded by $C$, or equivalently, if $L(\xi) \le C$. In \cite[Theorem 3.2]{GP}, Garnett and Papadimitriakis observed that for any almost isometric ray,
$F$ possesses an angular derivative at $\xi$ in the sense of Carath\'eodory. In particular, the radial boundary value $F(\xi) = \lim_{w \in \gamma, \, w \to \xi } F(w)$ exists and belongs to the unit circle.

One may also consider the {\em radial lag function} of $\gamma$ given by
\begin{equation}
    \label{rlag}
L^{\rad}(w) = d_h(z, w) - \bigl ( d_h(0, F(w)) - d_h(0, F(z)) \bigr ), \qquad w \in \gamma. 
\end{equation}
The same reasoning as above shows that $L^{\rad}(w)$ increases as $d_h(z, w)$ increases. 
We say that the geodesic ray $\gamma$ is $C$-{\em good} if 
$$
L^{\rad}(\xi) \, := \, \sup_{w \in \gamma} \bigl \{ d_h(z, w) -  \bigl ( d_h(0, F(w)) - d_h(0, F(z) \bigr ) \bigr \} \, \leq \, C.
$$
In other words, as one moves along $\gamma$ at unit hyperbolic speed, the image point efficiently moves towards the unit circle. By the triangle inequality, $L^{\rad}(w) \ge L(w)$, so a $C$-good ray is also a $C$-almost isometric ray. To understand the extent as to which the converse holds, we use the following lemma:

\begin{lemma}
\label{lag-and-rlag}
Let $F$ be an analytic self-mapping of the unit disk, $z \in \mathbb{D}$ be a point in the unit disk and  $I_z$ be the arc on the unit circle centered at $z/|z|$ with $ m (I_z) = 1-|z|$. Then,
 \begin{equation}
\label{eq:lag-and-rlag}
\log |F'(\xi)| - \log \frac{1-|F(z)|}{1-|z|} = L^{\rad}(\xi) + O(1), \qquad \xi \in I_z, \quad |F'(\xi)| < \infty, 
\end{equation}
and
 \begin{equation}
\label{eq:lag-and-rlag2}
L^{\rad}(\xi) = L(\xi) + 2 \log k(\xi) + O(1), \qquad \xi \in I_z, \quad |F'(\xi)| < \infty,
\end{equation}
where $L$ and $L^{\rad}$ are the lag functions associated to the geodesic ray $[z, \xi)$ and $k(\xi)$ is the smallest real number $\ge 1$ such that  $F(\xi) \in k  I_{F(z)}$.
\end{lemma}

In the off-chance that $z=0$, we take $I_z$ to be the unit circle $\partial \mathbb{D}$, while if $F(z) = 0$, then we take $I_{F(z)} = \partial \mathbb{D}$ and $k = 1$.

\begin{proof}
Fix $\xi \in I_z$. For a point $w$ on the geodesic ray $[z,\xi)$, we have
$$
d_h(0, w) = d_h(0,z) + d_h(z, w) - O(1).
$$
Taking the difference of the equations
$$
\log \frac{1-|F(w)|}{1-|w|} + O(1) \, = \, d_h(0, w) - d_h(0, F(w)) + O(1)
$$
and 
$$
\log \frac{1-|F(z)|}{1-|z|} + O(1) \, = \, d_h(0, z) - d_h(0, F(z)) + O(1),
$$
we get
$$
\log \frac{1-|F(w)|}{1-|w|}  - \log \frac{1-|F(z)|}{1-|z|} = L^{\rad}(w) + O(1).
$$
The estimate (\ref{eq:lag-and-rlag}) now follows after taking the limit as $w \to \xi$ along $\gamma$ and the definition of the angular derivative.

The hyperbolic distance between two points $z_1, z_2 \in \mathbb{D}$ is
$$
d_h(z_1, z_2) = d_h(|z_1|, |z_0|) + d_h(|z_0|, |z_2|) + O(1),
$$
where $z_0$ is the point on the geodesic segment $[z_1,z_2]$ that is closest to the origin.
Applying this to the current situation, we see that if $w$ is a point on the geodesic ray $[z, \xi)$ close to $\xi$, then
\begin{align*}
d_h(F(z), F(w)) & = \log k + \bigl \{ \log k + d_h(|F(z)|, |F(w)|) \bigr \} + O(1) \\
& = 2\log k + \bigl \{ d_h (0, |F(w)|) - d_h(0,|F(z)|) \bigr \} + O(1),
\end{align*}
which leads to (\ref{eq:lag-and-rlag2}) after some algebraic manipulation and taking $w \to \xi$.
\end{proof}

We now record two simple consequences of Lemma \ref{lag-and-rlag} that will be used throughout the paper.
Corollary \ref{a-lower-bound-angular-derviative} below gives a lower bound for the angular derivative at a point $\xi \in I_z$, while Corollary \ref{angularderivativegoodgeodesicrays} says that this lower bound is essentially sharp when $[z,\xi)$ is a good geodesic.

\begin{corollary}
\label{a-lower-bound-angular-derviative}
Let $F$ be an analytic self-mapping of the unit disk, $z \in \mathbb{D} \setminus \{ 0 \}$ be a point in the unit disk and  $I_z$ be the arc on the unit circle centered at $z/|z|$ with $ m (I_z) = 1-|z|$. Then,
\begin{equation}
\label{eq:a-lower-bound-angular-derviative}
 |F'(\xi)| \ge C \cdot \frac{1-|F(z)|}{1-|z|}, \quad \xi \in I_z,
\end{equation}
where $C$ is a universal constant.
\end{corollary}

\begin{corollary}
\label{angularderivativegoodgeodesicrays}
Let $F$ be an analytic self-mapping of the unit disk and suppose that $z \in \mathbb{D} \setminus \{ 0\}$ and $\xi \in I_z$.
The following statements are equivalent:
\begin{enumerate}
\item[{\em (1)}] The geodesic ray $[z, \xi)$ is $C$-good for some constant $C > 0$.

\item[{\em (2)}] There exists a constant $C_1 \ge 1$ such that
\begin{equation}
    \label{nouu}
     |F'(\xi)| \le C_1 \cdot \frac{1-|F(z)|}{1-|z|}.
\end{equation}

\item[{\em (3)}] The  geodesic ray $[z, \xi)$ is $C_2$-almost isometric and $F(\xi) \in K  I_{F(z)}$ for some constants $C_2, K > 0$.
\end{enumerate}
If one uses the optimal constants for $C, C_1, C_2$ and $K$, then $C$, $\log C_1$ and $C_2 + 2 \log^+ K$ differ by at most a universal constant.
\end{corollary}

\subsection{Counting pre-images}

Before continuing, we recall a well-known inequality due to Littlewood:

\begin{lemma}
\label{littlewood}
Let $F$ be an analytic self-mapping of the unit disk with $F(0) = 0$. For any $v \in \mathbb{D} \setminus \{0\}$, we have
$$
\sum_{F(u) = v} \log \frac{1}{|u|}  \le \log \frac{1}{|v|}.
$$
In particular, there exists a universal constant $C>0$ such that 
\begin{equation}
\label{eq:littlewood}
\sum_{F(u) = v} e^{- d_h (0, u)} \leq C e^{- d_h (0, v)}, \qquad d_h(0, v) > 1.
\end{equation}
\end{lemma}

For a proof, see  \cite[Section 10.4]{shapiro}. M\"obius invariance of hyperbolic distance shows:

\begin{corollary}
\label{counting-preimages}
There exists a universal constant $C>0$ such that for any analytic self-mapping $F$ of the unit disk and any pair of points $u_0, v \in \mathbb{D}$ with 
$d_h(F(u_0), v) > 1$, we have 
\begin{equation}
\label{eq:counting-preimages}
\sum_{F(u) = v} e^{- d_h (u, u_0)} \leq C e^{- d_h (v, F(u_0))}. 
\end{equation}
\end{corollary}

\begin{remark}
A little thought shows that the estimate (\ref{eq:littlewood}) holds for all $v \in \mathbb{D}$ with a constant $C = C_F$ that depends on the number of zeros of $F$ in $B_h(0,1)$, counted with multiplicity. Similarly, (\ref{eq:counting-preimages}) holds for all $v \in \mathbb{D}$ with a constant $C = C_F$ that depends on the number of times $F$ attains the value $F(u_0)$ in $B_h(F(u_0), 1)$.
\end{remark}

\subsection{Equivalence of (1) and (2) in Theorem \ref{main-thm}}

We now relate the area condition (\ref{eq:area-condition}) to the existence of good geodesic rays, thereby showing the equivalence of Statements (1) and (2) in Theorem \ref{main-thm}. To this end, we show the following lemma:

\begin{lemma}
\label{condition2}
Let $F$ be an inner function. The following conditions are equivalent:

{\em (i)} There exists a constant $c>0$ such that $A_h (F(B_h (z,R) ) > c\, A_h (B_h (z, R))$ for any 
$z \in \mathbb{D}$ and $R>1$.

{\em (ii)} There exist constants $C, \delta > 0$ such that for any $z \in \mathbb{D}$, we have
$$
\omega \bigl(z,  \{ \xi \in \partial \mathbb{D} : [z, \xi) \text{ is a $C$-almost isometric ray}\},  \mathbb{D} \bigr) \ge \delta.
$$

{\em (iii)} There exist constants $C, \delta > 0$ such that for any $z \in \mathbb{D}$, we have
$$
\omega \biggl( z, \biggl \{ \xi \in \partial \mathbb{D} : |F'(\xi)| < C \cdot \frac{1-|F(z)|}{1-|z|} \biggr \} ,  \mathbb{D} \biggr) \ge \delta.
$$

The constants $C,\delta$ depend only on $c$ and vice versa.
\end{lemma}

\begin{proof}
 (i) $\Rightarrow$ (ii).  
Recall that the hyperbolic area of a disk grows exponentially in the hyperbolic radius,
 $$
 A_h (B_{h}(z,R)) \sim \pi e^{R}, \qquad \text{as }R \to \infty , 
 $$
uniformly in $z \in \mathbb{D}$. We choose $C = 1 + \log(1/c)$ so that
 $$
 A_h (B_{h}(z,R-C)) \le (c/2) A_h (B_{h}(z,R)), \quad z \in \mathbb{D},
 $$
 for any $R > R_0$ sufficiently large.
 
For a geodesic ray $[z, \xi)$, $\xi \in \partial \mathbb{D}$, we write $z_{\xi, R} \in [z, \xi)$ for the point with $d_h(z, z_{\xi, R}) = R$. Consider the set $A(z,R)$ of points $\xi \in \partial \mathbb{D}$ such that 
$$d_h(F(z), F(z_{\xi, R})) \geq R-C .$$ 
Assumption (i) and the Schwarz lemma imply that $\omega (z, A(z,R), \mathbb{D} ) \gtrsim c$. As the sets $A(z,R)$ are decreasing in $R$, their intersection $A(z) = \bigcap_{R > 1} A(z,R)$ also satisfies $\omega (z, A(z), \mathbb{D} ) \gtrsim c$. Since all the geodesic rays $[z,\xi)$ with $\xi \in A(z)$ are $C$-almost isometric, (ii) holds. 
 
\medskip

 (ii) $\Rightarrow$ (i). 
Define $\GP(z)$ as the set of points $\xi \in \partial \mathbb{D}$ for which the geodesic ray $[z, \xi)$ is $C$-almost isometric. 
Let $\mathcal K \subset \mathbb{D}$ be the union of geodesic rays $[z,\xi)$ with $\xi \in \GP (z)$ and  $\mathcal K' \subset \mathcal K$
be the subset of points $w \in \mathcal{K}$ with $|(\lambda_F/\lambda)(w)| > 1/2$. Since
$$
d_h(F(w), F(z)) \leq \int_{[z,w)} \frac{\lambda_F (u)}{\lambda (u)} \cdot \lambda(u)|du|, \qquad w \in \mathbb{D},
$$
 for any $\xi \in \GP (z)$, the intersection $(\mathcal K \setminus \mathcal K') \cap [z,\xi)$ can have hyperbolic length at most $2C$. 
 Consequently, for $R > 2C + 1$, the ratio
 $$
\frac{A_h (\mathcal K' \cap B_{h}(z,R))}{A_h (B_{h}(z,R))} \gtrsim e^{-2C} \cdot \delta
$$
is bounded below by a constant that only depends on $C$ and $\delta$.
To prove (i), it is enough to show that the quotient
$$
\frac{ A_h (F(\mathcal K' \cap B_{h}(z,R)))}{ A_h (B_{h} (z,R))}
$$
is bounded below by a constant independent of $z$ and $R$.

Since $F$ can only contract the hyperbolic metric on $\mathcal K'$ by a factor of at most $2$, the hyperbolic area of $F(\mathcal K'  \cap B_{h}(z,R))$
counted with multiplicity is comparable to the hyperbolic area of $B_{h}(z, R)$. Therefore, to prove (i), it suffices to show that there is a uniformly bounded number of collisions. To that end, suppose that points $u_1, u_2, \dots, u_N \in \mathcal K'$ have the same image under $F$, that is, 
$$F(u_1) = F(u_2 ) = \dots = F(u_N) = v.$$
Since  $\mathcal K$ was defined as the union of $C$-almost isometric rays, we have
$$
d_h(z, u_i) - d_h(F(z),v) \leq C, \qquad i =1, 2, \dots, N.
$$
Corollary \ref{counting-preimages} then implies that $N$ is uniformly bounded. The proof is complete.

\medskip

 (ii) $\Rightarrow$ (iii). 
As before, let $\GP (z)$ be the set of points $\xi \in \partial \mathbb{D}$ such that the geodesic ray $[z, \xi)$ is $C$-almost isometric. Below, we show that we can remove a set of points $\xi \in \partial \mathbb{D}$ of small harmonic measure seen from $z$, so that the geodesic rays $[z , \xi)$ corresponding to the remaining points will be good.
To this end, choose a sufficiently large constant $K = K(\delta) > 1$ so that
$$
\omega (F(z), \partial \mathbb{D} \setminus K  I_{F(z)}, \mathbb{D})  < \delta/2.
$$
By the L\"owner's lemma, $w (F(z), E , \mathbb{D}) = w (z, F^{-1} (E) , \mathbb{D})$, $E \subset \partial \mathbb{D}$, we have
$$\omega (z,  F^{-1}(\partial \mathbb{D} \setminus K  I_{F(z)}), \mathbb{D} ) < \delta/2.$$
Consequently, the assumption shows that $\omega (z, E(z), \mathbb{D}) \geq \delta /2$, where $E(z)$ is the set of points $\xi \in \GP (z)$ with $F(\xi) \in K  I_{F(z)}$. According to Corollary \ref{angularderivativegoodgeodesicrays}, for any $\xi \in E(z)$, the geodesic ray $[z, \xi)$ is $C_1 = C_1 (\delta , C)$-good and hence 
$$
|F'(\xi)| \leq e^{C_2}\, \frac{1-|F(z)|}{1-|z|}, 
$$
where $C_2 = C_1 + O(1)$.

\medskip

 (iii) $\Rightarrow$ (ii). This implication is obvious, since every $C$-good ray is a 
$C$-almost isometric ray. 
\end{proof}

\section{BMO estimates}

In this section, we show the equivalence of Conditions (2), (3), (4) and (5). The proofs involve a stopping time argument that we describe in Section \ref{sec:stopping-time-region}.

\subsection{A stopping time argument}
\label{sec:stopping-time-region}

Suppose $F$ is an analytic self-mapping of the unit disk.
Fix a constant $M>0$ and an arc $I \subset \partial \mathbb{D}$. We examine the dyadic decomposition of $I$.
Let $ \{ J_k \}$ be the collection of maximal dyadic arcs in $I$ for which
$$
\log \frac{1-|F(z_{J_k})|}{1-|z_{J_k}|} \ge M + \log \frac{1-|F(z_{I})|}{1-|z_{I}|}.
$$
We refer to the arcs $\{ J_k \}$ as arcs of generation 1. The stopping time region of generation 1 is defined as $\Omega_I = Q_I \setminus \bigcup Q_{J_k}$.
We record three simple observations:

\begin{itemize}
\item
By the maximality of the arcs $J_k$ and the Schwarz lemma, one has:
\begin{equation}
\label{eq:maximality-of-the-arcs}
\log \frac{1-|F(z_{J_k})|}{1-|z_{J_k}|} = M + \log \frac{1-|F(z_{I})|}{1-|z_{I}|} + O(1).
\end{equation}

\item
Corollary \ref{a-lower-bound-angular-derviative} tells that
\begin{equation}
\label{eq:lower-bound-on-good-intervals}
\log |F'(\xi)| - \log \frac{1-|F(z_{I})|}{1-|z_{I}|} \ge M - O(1),
\end{equation}
for any $\xi \in \bigcup_k J_k$.

\item 
We have
\begin{equation}
\label{estimateonregion}
\log \frac{1-|F(z)|}{1-|z|} \le M + \log \frac{1-|F(z_{I})|}{1-|z_{I}|} + O(1), \quad  z \in \Omega_I.
\end{equation}
\end{itemize}

We may repeat the construction in each arc of generation 1, that is, replace the initial arc $I$ by arcs of generation $1$ and obtain the collection of arcs of generation $2$. Continuing inductively, the stopping time region $\Omega_I^{(n)}$ of generation $n$ is defined as $Q_I \setminus \bigcup Q_{J_k^{(n)}}$ where we remove all Carleson squares associated to arcs $\{J_k^{(n)} \}$ of generation $n$.

\subsection{Equivalence of (2), (3), (4) and (5) in Theorem \ref{main-thm}}
\label{sec:equivalence-2345}

\begin{proof}[$(2) \Rightarrow (3)$]

Let $E= \{\xi \in \partial \mathbb{D}: |F'(\xi )| < \infty \}$. Condition (2) implies that 
$$
\omega (z,E,\mathbb{D} ) \geq \delta, \qquad \text{for any }z \in \mathbb{D}. 
$$
As the linear density of the set $E$ at any point of the unit circle is bounded below by a fixed multiple of $\delta$, the Lebesgue density point theorem tells us that $m(E) = 1$. 

For an arc $I \subset \partial \mathbb{D}$ in the unit circle, let $G(I)$ denote the set of points $\xi \in I$ such that the geodesic ray $[z_I , \xi)$ is $C$-good. 
In view of Corollary \ref{angularderivativegoodgeodesicrays}, we may rephrase Condition (2) as: there exist constants $C>0$ and $\delta >0 $ such that 
\begin{equation}
    \label{mesuragoodgeodesicray}
    m ( G(I) ) \geq \delta \, m(I), 
\end{equation}
for any arc $I \subset \partial \mathbb{D}$. 

We now run the stopping time argument described in the previous section with a large constant $M$ to be determined later. In view of (\ref{eq:lower-bound-on-good-intervals}), if the parameter $M$ is sufficiently large,
then by \eqref{mesuragoodgeodesicray}, the total length of the arcs of generation 1 is
$$
\sum m(J_k ) \le (1-\delta) m(I).
$$
An inductive argument shows that the total length of the arcs $\{J_k^{(n)} \}$ of generation $n$ is
$$
\sum_k m( J_k^{(n)} ) \le (1-\delta)^n m( I).
$$
By \eqref{estimateonregion}, in the complement of the arcs of generation at most $n$, we have
$$
\log |F'(\xi)| - \log  \frac{1-|F(z_{I})|}{1-|z_{I}|} \le n(M + O(1)), \quad \xi \in I \setminus \bigcup_k J_k^{(n)}, \quad n=1,2,\ldots 
$$
This shows that $\log |F'| \in L^1 (\partial \mathbb{D})$ and there exists a constant $C(M, \delta) >0$ such that 
$$
\frac{1}{m(I)} \int_I \log |F'| dm \leq \log \frac{1-|F(z_I)|}{1-|z_I|} + C(M, \delta).
$$
As we have seen in Remark 2 after Theorem \ref{main-thm} in the introduction, the lower bound in (\ref{mean}) is automatic. The same argument gives the seemingly stronger conclusion that $\log |F'| \in \BMO (\partial \mathbb{D})$, as discussed in Remark 3 after Theorem \ref{main-thm}. 
\end{proof}

\begin{proof}[$(3) \Rightarrow (4)$]
Consider the harmonic function
$$
E(z) = \int_{\partial \mathbb{D}} \log |F' (\xi)| \, \frac{1-|z|^2}{|\xi - z|^2} dm (\xi), \quad z \in \mathbb{D}.
$$
From the Schwarz lemma and the lower bound in \eqref{mean}, it follows that  $$\log |F' (z)| \leq C_1 \cdot E(z) + C_2, \quad z \in \mathbb{D},$$ for some constants $C_1, C_2 \in \mathbb{R}$. As $\log|F'(z)|$ has a harmonic majorant in the unit disk, $F'$ is in the Nevanlinna class and we can consider
its inner-outer factorization $F' = I_{F'} \cdot O_{F'}$.

A further inspection of the estimate \eqref{mean} tells us that $E$ is a harmonic function in the Bloch space and 
\begin{equation}
\label{bloch}
E(z) = \frac{1}{m(I_z)} \int_{I_z} \log |F'| dm + O(1). 
\end{equation}
See \cite[Lemmas J.1 and J.2]{garnett-marshall}. Since $\log |O_{F'} (z)| = E(z)$, $z \in \mathbb{D}$, the estimate \eqref{outer} follows from \eqref{bloch} and \eqref{mean}. 
\end{proof}

\begin{proof}[$(4) \Rightarrow (5)$]
The condition $\eqref{outer}$ may be written as
\begin{equation}
\label{ob}
  \biggr| \log |O_{F'} (z)| - \log \frac{1-|F(z)|}{1-|z|} \biggr| \leq \log C , \quad z \in \mathbb{D} .   
\end{equation}
By the Schwarz lemma, there exists a universal constant $C_1 >0$ so that
$$
\biggr| \log \frac{1-|F(z)|}{1-|z|} - \log \frac{1-|F(w)|}{1-|w|} \biggr| \leq C_1,
$$
if $d_h (z, w) \leq 1$. As a result,
$$
\bigr| \log |O_{F'} (z)| - \log |O_{F'} (w)| \bigr| \leq C_1 + 2 \log C,
$$
if $d_h (z, w) \leq 1$. From here, it is not difficult to see that $\log O_{F'}$ lies in the Bloch space. Consequently, by \cite[Lemma J.1]{garnett-marshall},
$$
\log |O_{F'} (z) | =  \frac{1}{m(I_z)} \int_{I_z} \log |F'| dm + O(1), \quad z \in \mathbb{D}, 
$$
where $O(1)$ is a quantity bounded by a constant independent of $z \in \mathbb{D}$. Combining with the estimate \eqref{ob}, we get
\begin{equation}
    \label{nova}
    \log \frac{1-|F(z)|}{1-|z|} =  \frac{1}{m(I_z)} \int_{I_z} \log |F'| dm + O(1), \qquad z \in \mathbb{D}.
\end{equation}
On the other hand, by Corollary \ref{a-lower-bound-angular-derviative}, there is a universal constant $L > 0$ so that
\begin{equation}
    \label{A}
\log \frac{1-|F(z)|}{1-|z|} \leq \log |F' (\xi)| + L, \qquad \xi \in I_z.
\end{equation}
We pick $M= M(\varepsilon)>0$ sufficiently large so that the estimates \eqref{nova} and \eqref{A} guarantee that 
\begin{equation}
    \label{nou}
    m \biggl ( \biggl \{ \xi \in I_z : \log |F' (\xi)| \ge \log \frac{1-|F(z)|}{1-|z|} +M \biggr \} \biggr ) \leq \varepsilon \, m(I_z), \quad z \in \mathbb{D}. 
\end{equation}

Fix an arc $I \subset \partial \mathbb{D}$ and run the stopping time argument  with the parameter $M_1 = M + L$ to find pairwise disjoint arcs $J_k \subset I$ with
$$
\log \frac{1-|F(z_{J_k} )|}{1-|z_{J_k}|} \geq M_1 + \log \frac{1-|F(z_I)|}{1-|z_I|}.
$$
Since
$$
\log|F'(\xi)| \geq M + \log \frac{1-|F(z_I)|}{1-|z_I|}, \qquad \xi \in \bigcup_k J_k,
$$
by Corollary \ref{a-lower-bound-angular-derviative}, the total length of the arcs $\sum m(J_k) \leq \varepsilon \, m (I)$. By the assumption \eqref{outer} and the estimate \eqref{estimateonregion}, we have
$$
|F'(z)| \, \leq \, |O_{F'} (z)| \, \leq \, C \, \frac{1 - |F(z)|}{1-|z|} \, \leq \, C_1  e^{M_1} \, \frac{1-|F(z_I)|}{1-|z_I|},
$$
for any $z \in Q_I \setminus \bigcup_k Q_{J_k}$.
\end{proof}

\begin{proof}[$(5) \Rightarrow (2)$]
Let $I \subset \partial \mathbb{D}$ be an arc. In view of the estimate \eqref{lipschitz}, the angular derivative satisfies 
$$
|F'(\xi)| \le C(\varepsilon) \cdot  \frac{1-|F(z_I)|}{1-|z_I|}, \qquad \xi \in I \setminus \bigcup J_k.
$$
Since $\sum m(J_k) < \varepsilon \, m(I)$, there exists $c_1= c_1 (\varepsilon) >0$ such that
$$
\omega \biggl (z_I , \biggl \{\xi \in \partial \mathbb{D} : |F'(\xi)| \leq C(\varepsilon) \cdot \frac{1-|F(z_I)|}{1-|z_I|} \biggr \} , \mathbb{D} \biggr ) \geq c_1. 
$$
The proof is complete.
\end{proof}

\section{Carleson condition for critical points}

In this section, we show that APHA mappings can be described using various Carleson measure conditions. The plan is to show the implications
(3) $\Rightarrow$ (6) $\Rightarrow$ (7) $\Rightarrow$ (4) of Theorem \ref{main-thm}. 

\subsection{Background on M\"obius distortion}

Recall that the M\"obius distortion
$$
\mu(z) = 1 - \frac{(1-|z|^2) |F'(z)|}{1-|F(z)|^2}
$$
of an analytic self-mapping $F$ of the unit disk measures how much $F$ deviates from an automorphism of the disk at a point $z \in \mathbb{D}$.
The following lemma follows from a normal families argument:

\begin{lemma}
\label{mobius-distortion-basics}
Let $F$ be an analytic self-mapping of the unit disk and $\mu$ be its Möbius distortion. 

{\em (i)} There exists a constant $\mu_0 >0$ such that if $F' (c)=0$, then $\mu(w) > \mu_0 $ for any $w \in B_{h}(c,1)$.

{\em (ii)} For any $R > 0$, there exists a sufficiently large constant $N = N(R)$ so that if $B_{h}(z, 1)$ contains $N$ critical points of $F$ counted with multiplicity, then
$\mu (w) > 1/2$ for any  $ w \in B_{h}(z,R)$.
\end{lemma}

We will also use the following lemma which characterizes inner functions of finite entropy in terms of their M\"obius distortion:

\begin{lemma}
\label{finite-entropy-characterization}
An inner function $F$ has finite entropy if and only if its Mobius distortion $\mu$ satisfies 
\begin{equation}
\label{eq:mu-L1-bound}
\int_{\mathbb{D}} \mu(z) \, \frac{dA(z)}{1-|z|} < \infty.
\end{equation}
\end{lemma}

\begin{proof} Without loss of generality, we can assume that $F(0)=0$. By \cite[Lemma 3.3]{inner}, an inner function has finite entropy if and only if
\begin{equation*}
\lim_{r \to 1} \int_{|z|=r} \log \frac{1-|F(z)|^2}{1-|z|^2} \, |dz| < \infty,
\end{equation*}
in which case,
\begin{equation*}
\int_{\partial \mathbb{D}} \log |F'(z)| dm \, = \, \lim_{r \to 1} \, \frac{1}{2\pi} \int_{|z|=r} \log \frac{1-|F(z)|^2}{1-|z|^2} \, 
|dz|.
\end{equation*}
Applying Green's formula on the disk $\{z : |z| \leq r \}$ and taking $r \to 1$ shows that the above expression is equal to
\begin{equation*}
\frac{1}{2\pi} \int_{\mathbb{D}} \Delta \biggr( \log  \frac{1-|F(z)|^2}{1-|z|^2} \biggr)  \log \frac{1}{|z|} \, dA(z).
\end{equation*}
Since
$$
 \Delta \biggr( \log  \frac{1-|F(z)|^2}{1-|z|^2} \biggr) = \biggl ( \frac{2}{1-|z|^2} \biggr )^2 - \biggl ( \frac{2 |F'(z)|}{1-|F(z)|^2} \biggr )^2, \quad z \in \mathbb{D},
 $$
the above integral is comparable to
\begin{equation*}
  \int_{\mathbb{D}} \mu(z) \, \frac{dA(z)}{1-|z|},
 \end{equation*}
which proves the lemma.
\end{proof}

\subsection{Carleson-Newman Blaschke products}
A Blaschke product is called an {\em interpolating Blaschke product} if its zeros form an interpolating sequence for bounded analytic functions in $\mathbb{D}$. According to a celebrated result of L.~Carleson, a Blaschke product $B$ is interpolating if and only if its zeros are separated in the hyperbolic metric and there exists a constant $C>0$ such that
      \begin{equation}
          \label{CNBP}
          \sum_{z \in Q_I: \, B(z) = 0} (1-|z|) \leq C m(I),
      \end{equation}
      for any Carleson square $Q_I \subset \mathbb{D}$. A proof can be found in \cite[Chapter VII]{garnett}.

      Recall from the introduction that a Blaschke product $B$ is  a {\em Carleson-Newman Blaschke product} if (\ref{CNBP}) holds, without requiring separation.
 It is well-known that Carleson-Newman Blaschke products are precisely those Blaschke products which can be factored into a finite product  of interpolating Blaschke products.
 
      In the following lemma, we gather several elementary properties of Carleson-Newman Blaschke products that will be used below:
      
     \begin{lemma}
\label{carleson-newman}
Suppose $B$ is a Carleson-Newman Blaschke product with zeros $\{z_n \}$. 

{\em (a)} For any $\varepsilon > 0$, there exists a $\delta > 0$ such that 
\begin{equation}
\label{CN}
\{ z \in \mathbb{D}: |B(z)| < \delta \} \subset \bigcup B_{h} (z_n , \varepsilon).     
\end{equation}

{\em (b)} For any $\rho >0$, there exists an $N = N(B, \rho)$ such that any disk of hyperbolic radius $\rho$ contains at most $N$ zeros of $B$, counted with multiplicity. 

{\em (c)} There exist constants $\delta, \varepsilon >0$ such that for any $z \in \mathbb{D}$ there exists $w \in \mathbb{D}$ with $d_h (w,z) \leq \delta$ such that $|B(w)| \geq \varepsilon$. 
Conversely, the existence of such constants implies that $B$ is a Carleson-Newman Blaschke product.
\end{lemma}

\begin{proof}[Proof]
For interpolating sequences, (a) is just \cite[Lemma 1]{KL}. If $B = B_1 B_2 \cdots B_k$ is a finite product of interpolating Blaschke products, with constants
$\delta_1, \delta_2, \dots, \delta_k$ corresponding to $\varepsilon$ for which the inclusion \eqref{CN} holds, then \eqref{CN} also holds for $B$ with
$\delta = \delta_1 \delta_2 \dots \delta_k$.

From \eqref{CNBP}, it is clear that any disk contained in the top part of a Carleson square may contain a uniformly bounded number of zeros of $B$. This proves (b) for disks of sufficiently small hyperbolic radius. To see the general case, it is enough to notice that any disk of a fixed hyperbolic radius can be covered by a bounded number of disks of small hyperbolic radius.
Finally, (c) can be found in \cite[Theorem 2.2]{MN}.
\end{proof}

For future reference, we record the following corollary:

\begin{corollary}
\label{good-annuli}
Suppose $B$ is a Carleson-Newman Blaschke product with zeros $\{z_n \}$. For any $\rho > 0$, there exist constants $\delta, m > 0$, an integer $N \in \mathbb{N}$ and a collection of round annuli
$$A_n = \{ z \in \mathbb{D}: r_n < d_h(z_n, z) < R_n \}$$
satisfying the following conditions:

\begin{enumerate}
\item[$(1)$] $A_n \subset  \{ z \in \mathbb{D}: \rho < d_h(z_n, z) < 2 \rho \}$.
\item[$(2)$] $R_n/r_n \ge 1 + m$.
\item[$(3)$] $|B(z)| > \delta$ for $z \in A_n$.
\item[$(4)$] Any point in the unit disk is contained in at most $N$ of the annuli $A_n$.
\end{enumerate}
\end{corollary}

\begin{proof}
By Lemma \ref{carleson-newman}(b), there exists an integer $N = N(B)$ so that any disk of hyperbolic radius $2\rho$ contains at most $N$ zeros of $B$, which a fortiori implies (4).
This allows us to choose the annuli $A_n$ which satisfy the conditions (1), (2) with $m \asymp 1/N$ and contain no zeros of $B$. By replacing each annulus with the middle third sub-annulus, we can guarantee that
\begin{enumerate}
\item[$(3')$] $d_h \bigl (A_n, \{ z: B(z) = 0 \} \bigr ) \gtrsim 1/N$.
\end{enumerate}
An application of Lemma \ref{carleson-newman}(a) provides a $\delta$ so that (3) holds.
\end{proof}

\subsection{Equivalence of (3), (4), (6) and (7) in Theorem \ref{main-thm}}

\begin{proof}[$(3) \Rightarrow (6)$]
Fix a point $z \in \mathbb{D}$ and let $[z,z_1]$ denote the geodesic arc joining $z, z_1 \in \mathbb{D}$. From
$$
d_h (F(z_1), F(z)) \leq \int_{[z, z_1]} \frac{2|F'(w)|}{1-|F(w)|^2} \, |dw|,
$$
it is clear that 
$$
\int_{[z, z_1]} \mu (w) \, \frac{2|dw|}{1-|w|^2}  \leq d_h (z_1, z) - d_h (F(z_1), F(z)).
$$
By the triangle inequality, we have
$$
d_h (F(z_1), F(z)) \, \ge \, d_h (|F(z_1)|, |F(z)|) \, \ge \, d_h (0, |F(z_1)|) - d_h (0, |F(z)|).
$$
In particular, if $z_1$ belongs to a geodesic ray which connects $z$ to a point $\xi \in I_z$, then
\begin{align*}
\int_{[z, z_1]} \mu (w) \, \frac{2|dw|}{1-|w|^2} & \leq \biggl \{ \log \frac{1-|z|}{1-|z_1|} + O(1) \biggr \} - \biggl \{ \log \frac{1-|F(z)|}{1-|F(z_1)|} + O(1) \biggr \} \\
& \leq \log \frac{1-|F(z_1)|}{1-|z_1|} - \log \frac{1-|F(z)|}{1-|z|} + O(1).
\end{align*}
Taking $z_1 \to \xi$ along $[z, \xi)$, we obtain
$$
\int_{[z, \xi)} \mu (w)\, \frac{2|dw|}{1-|w|^2} \leq \log |F'(\xi)| - \log \frac{1-|F(z)|}{1- |z|} + O(1).
$$
It remains to integrate over $\xi \in I$ and invoke the estimate \eqref{mean}.
\end{proof}

\begin{proof}[$(6) \Rightarrow (7)$]
By Lemma \ref{mobius-distortion-basics}(ii), the Carleson condition for the M\"obius distortion given in the estimate \eqref{estimatedistortion} implies that the number of critical points of $F$ (counted with multiplicity) in any disk of hyperbolic radius 1 is uniformly bounded above. In other words, any point in the unit disk is contained in a uniformly bounded number of disks $B_h (c, 1)$ with $c \in \crit F$.
The Carleson estimate for critical points
$$
\sum_{c \in  Q_I:\, F'(c)=0} (1 - |c|) \lesssim m(I),
$$
now follows from Lemma \ref{mobius-distortion-basics}(i) and \eqref{estimatedistortion}. 

To see that $F$ satisfies Condition (7), it remains to show that $F$ is a maximal Blaschke product. From Lemma \ref{finite-entropy-characterization}, we know that $F$ has finite entropy. According to \cite[Corollary 2.2]{inner}, $F$ is a maximal Blaschke product if and only if the inner factor of $F'$ is a Blaschke product. 

We proceed by contradiction: if $F$ is not a maximal Blaschke product, then $F'$ has a nontrivial singular inner factor. Let $\sigma = \sigma(F')$ be the corresponding singular measure. 
As explained in \cite[Theorem 1.2]{IK22},
for $\sigma$ a.e.~$\xi \in \partial \mathbb{D}$, the non-tangential limit $\lim_{r \to 1 } F(r \xi)$ exists and lies in the open unit disk $\mathbb{D}$. In particular, by Lemma \ref{lambda-lemma1}, $\mu (z) \to 1$
as $z \to \xi$ non-tangentially. As a result of this and Lemma \ref{lambda-lemma2}, the measure $\mu(z) (1-|z|)^{-1} dA(z)$ violates the Carleson condition in the Carleson squares $Q_I$ associated to small arcs $I$ centered at $\xi$ of lengths tending to $0$. 
\end{proof}

\begin{proof}[$(7) \Rightarrow (4)$]
Since $C$ is a Blaschke sequence, the derivative of the maximal Blaschke product $F$ with critical set $C$ is in the Nevanlinna class by
\cite[Theorem 4.4]{kraus}. Let $F' = B O$ be the inner-outer factorization of $F'$.
By Lemma \ref{carleson-newman}(c), there exist constants $\varepsilon >0$ and $\delta >0$ such that for any $z \in \mathbb{D}$, there exists a $w \in \mathbb{D}$ with $d_h (z,w) \leq \delta$  such that $|F'(w)| \geq \varepsilon |O (w)|$. As $|O(z)| \asymp |O(w)|$,
$$
 \frac{1-|F(z)|}{1-|z|} \, \asymp \,  \frac{1-|F(w)|^2}{1-|w|^2} \, \ge \, |F'(w)| \, \ge \, \varepsilon  |O (w)| \, \asymp \, |O (z)|.
$$
The estimate in the other direction always holds in view of Dyakonov's inequality \eqref{eq:dyakonov}.
\end{proof}

\section{Aleksandrov-Clark measures}

In this section, we show Theorem \ref{ACofAPHAM}.
Given a positive Borel measure $\sigma$ on the unit circle, recall that $H[\sigma]$ is defined as 
$$
H[\sigma](z) = \int_{\partial \mathbb{D}} \frac{d\sigma(\xi)}{|\xi - z|^2}, \qquad z \in \overline{ \mathbb{D}}.
$$
Let $u$ be the Poisson extension of $\sigma$ to $\mathbb{D}$. Note that $H[\sigma] (z) = u(z)/ (1-|z|^2)$, $z \in \mathbb{D}$. For a Carleson square $Q=Q_I$, we denote its center by $z_Q = z_I$. By Harnack's inequality, for any Carleson square $Q$, we have
$$
\frac{1-|z|^2}{1-|z_Q|^2} \, \lesssim \, \frac{u(z)}{u(z_Q)} \, \lesssim \, \frac{1-|z_Q|^2}{1-|z|^2}, \qquad z \in Q.
$$
In particular, $H[\sigma]$ is bounded below in $\mathbb{D}$ by a positive constant and for any Carleson square $Q$, we have
\begin{equation}
\label{eq:u-monotonicity}
H[\sigma] (z_Q) \lesssim H[\sigma] (z), \qquad z \in \overline{ Q} .  
\end{equation}

Theorem \ref{ACofAPHAM} follows from the following more general result:

\begin{theorem}
\label{Htheorem}
Let $\sigma$ be a Borel measure on $\partial \mathbb{D}$ and let $H[\sigma]$ be defined by \eqref{Hdefinition}. The following statements are equivalent:

{\em (a)} The measure $\sigma$ is the Aleksandrov-Clark measure of some APHA mapping $F$ and some $\alpha \in \partial \mathbb{D}$.

{\em (b)} $\Delta \bigl ( \log H[\sigma] (z) \bigr ) (1-|z|^2) dA(z)$ is a Carleson measure.

{\em (c)} There exists a constant $C > 0$ such that for any Carleson square $Q \subset \mathbb{D}$ and any $M>1$, we have
$$
\sum_{Q_j \in \mathcal A(Q, M)} \ell (Q_j) \lesssim M^{-C}  \ell(Q),
$$
where $\mathcal A(Q, M)$ is the collection of maximal dyadic subsquares $\{Q_j\}$ of $Q$ such that $H[\sigma](z_{Q_j}) \ge M H[\sigma](z_Q)$.

{\em (d)} The function $\log H[\sigma] \in \BMO(\partial \mathbb{D})$ and
\begin{equation}
    \label{bmoac}
    \frac{1}{|I|} \int_I \log H[\sigma](\xi) dm (\xi) = \log H[\sigma](z_I) + O(1),
\end{equation}
uniformly over all arcs $I \subset \partial \mathbb{D}$.
\end{theorem}

In the proof of Theorem \ref{Htheorem}, we use the following technical lemma, whose proof will be given at the end of the section:

\begin{lemma}
\label{Hlemma}
Let $Q \subset \mathbb{D}$ be a Carleson square and  $\mathcal A(Q, M)$ be the collection of maximal dyadic subsquares $\{Q_j\}$ of $Q$ such that $H[\sigma](z_{Q_j}) \ge M H[\sigma](z_Q)$. 
Form the stopping time region $\Omega = Q \setminus \bigcup_{Q_j \in \mathcal A(Q, M)} Q_j$. Then, 
$$
\log M \cdot \sum_{Q_j \in \mathcal A(Q, M)} \ell(Q_j) - O(\ell(Q)) \le \int_{\Omega} \Delta \bigl ( \log  H[\sigma] (z) \bigr) \log \frac{1}{|z|} d A(z)
$$
$$\le \log M \cdot \ell(Q) + O(\ell(Q)).$$ 
\end{lemma}

\begin{proof}[Proof of Theorem \ref{Htheorem}]
(a) $\Leftrightarrow$ (b).  Let $u$ be the Poisson extension of $\sigma$ to $\mathbb{D}$. Since $u$ is the real part of $(\alpha + F)/ (\alpha - F)$ for some $\alpha \in \partial \mathbb{D}$, a quick computation shows that
$$
\frac{2|F'(z)|}{1-|F(z)|^2} = \frac{|\nabla u (z)|}{u(z)}, \qquad z \in \mathbb{D}.
$$
Since $\Delta ( \log u ) = - |\nabla u|^2/u^2$, we have
\begin{align*}
\frac{4}{(1-|z|)^2} - \frac{4 |F'(z)|^2}{ (1-|F(z)|^2)^2 } & = \Delta \bigl  ( - \log (1-|z|^2) + \log u(z)  \bigr ) \\
& = \Delta \bigl ( \log H[\sigma] (z) \bigr ),  \quad z \in \mathbb{D} . 
\end{align*}
In other words,
$$
4\biggl ( 1 - \frac{(1-|z|^2)^2 |F'(z)|^2}{(1-|F(z)|^2)^2} \biggr ) \frac{1}{1-|z|^2} = \Delta \bigl ( \log H[\sigma] (z) \bigr )  (1-|z|^2), \qquad z \in \mathbb{D}.
$$
By Condition (6) of Theorem \ref{main-thm}, $F$ is an APHA mapping if and only if
$$
\Delta \bigl ( \log H[\sigma] (z) \bigr )  (1-|z|^2) dA(z)
$$
is a Carleson measure. This proves the equivalence of statements (a) and (b).

\bigskip
\noindent
 (b) $\Rightarrow$ (c).
Invoking the lower bound in Lemma \ref{Hlemma} with $M_0$ in place of $M$, where $M_0 > 0$ is a large positive constant which will be fixed momentarily, we get
$$
\sum_{Q_j \in \mathcal A(Q, M_0)} \ell(Q_j) \le \frac{C}{\log M_0} \, \ell(Q).
$$
We fix $M_0 \ge e^{2C}$ so that
$$
\sum_{Q_j \in \mathcal A(Q, M_0)} \ell(Q_j) \le \frac{1}{2} \, \ell(Q).
$$
Notice that for any $k \in \mathbb{N}$, each square in $\mathcal A(Q, M_0^k)$ is contained in one of the squares in $\mathcal A(Q, M_0^{k-1})$. Consequently, if $M_0^{N-1} \leq  M \leq  M_0^N$ for some positive integer $N$, then
$$
\sum_{Q_j \in \mathcal A(Q, M)} \ell(Q_j) \, \leq \, \frac{1}{2^{N-1}}  \, \ell(Q) \, \lesssim \, M^{-C} \ell(Q),
$$
where $C = (\log_2 M_0)^{-1}$. This proves (c).

\bigskip
\noindent
 (c) $\Rightarrow$ (b).
Pick $M_0 > 0$ large enough so that for any Carleson square $Q$, we have
$$
\sum_{Q_j \in \mathcal A(Q, M_0)} \ell(Q_j)  \le \frac{1}{2} \, \ell(Q).
$$
Fix a Carleson square $Q_0$ and consider the domain
$$
\Omega_0(Q_0) = Q_0 \setminus \bigcup_{Q_j \in \mathcal A(Q_0, M_0)} Q_j.
$$
For each $Q_j \in \mathcal A(Q_0, M_0)$, we form an analogous domain
$$
\Omega(Q_j) = Q_j \setminus \bigcup_{Q_l \in \mathcal A(Q_j, M_0)} Q_l.
$$
Set $\Omega_1(Q_0) = \bigcup_{Q_j \in \mathcal A(Q_0, M_0)} \Omega(Q_j)$. Continuing inductively, we arrive at the decomposition $Q_0 = \bigcup_n \Omega_n(Q_0)$.

In view of the upper bound in Lemma \ref{Hlemma}, if $Q$ is one of the Carleson squares appearing in the construction, then
$$
\int_{\Omega(Q)} \Delta \bigl ( \log H[\sigma] (z) \bigr) \log  \frac{1}{|z|} dA(z) \leq C(M_0) \ell(Q).
$$
Summing over all Carleson squares appearing in the construction, we obtain
$$
\int_{Q_0} \Delta \bigl( \log H[\sigma] (z) \bigr) \log  \frac{1}{|z|} dA(z) \lesssim C(M_0) \ell(Q),
$$
as desired.

\bigskip
\noindent
(d) $\Rightarrow$ (c).
Let $Q=Q_I$ be a Carleson square and $\{ Q_{I_j} \}$ be the subsquares that make up $\mathcal A(Q, M)$. By (\ref{eq:u-monotonicity}) and (\ref{bmoac}), for any point $\xi \in \bigcup I_j$, we have
$$
 \log H[\sigma](\xi)  \ge
\frac{1}{m(I)} \int_{I} \log H[\sigma](\xi) dm(\xi)  + \log M - O(1).
$$
Since $\log H[\sigma] \in \BMO(\partial \mathbb{D})$, the John-Nirenberg Theorem tells us that there exists a universal constant $C>0$ such that 
$$
\sum_{Q_j \in \mathcal A(Q, M)} \ell(Q_j) \, \lesssim  \, e^{-C \log M} \ell(Q) \, = \, M^{-C} \ell(Q).
$$

\bigskip
\noindent
(c) $\Rightarrow$ (d).
Note that assumption (c) gives that $\sup_{r<1} H[\sigma] (r  \xi) < \infty$ for $m$ a.e.~$\xi \in \partial \mathbb{D}$. It follows that
$$
H[\sigma] (\xi) \, \leq \, \limsup_{r \to 1} H[\sigma] (r \xi) \, < \, \infty, \quad m \text{ a.e.~} \xi \in \partial  \mathbb{D}. 
$$
Pick $M > 1$ large enough so that $M^{-C} < 1/10$. By the estimate in (c), we have  
$$
m \bigl ( \bigl \{ \xi \in I : H[\sigma](\xi) \le M H[\sigma](z_I) \bigr \} \bigr ) \ge \frac{9}{10}\, m(I),
$$
for any arc $I \subset \partial \mathbb{D}$. Recall from (\ref{eq:u-monotonicity}) that $H[\sigma](z_I) \lesssim H[\sigma](\xi)$ for any $\xi \in I$. Consequently,
$$
m \bigl ( \bigl \{ \xi \in I : H[\sigma](z_I) \lesssim H[\sigma](\xi) \le M  H[\sigma](z_I) \bigr \} \bigr ) \ge \frac{9}{10}\, m(I),
$$
for any arc $I \subset \partial \mathbb{D}$. A classical result of Stromberg tells us that $\log H[\sigma] \in \BMO(\partial \mathbb{D})$ and
$$
\frac{1}{m(I)} \int_I \log H[\sigma] dm = \log H[\sigma](z_I) + O(M).
$$
See Exercise 4 on \cite[p.~261]{garnett}.
\end{proof}

It remains to show Lemma \ref{Hlemma}.

\begin{proof}[Proof of Lemma \ref{Hlemma}]
{\em Step 1.} We fix an $1/2 < r < 1$. Applying Green's identity to the functions $\log H[\sigma] (z)$ and $\log \frac{r}{|z|}$ in the domain
$\Omega(r) =  \Omega \cap \{z : |z| \leq r\}$, 
 we get
$$
\int_{\partial \Omega(r)} \log \bigl ( H[\sigma] (z) \bigr) \partial_{\bf n} \log \frac{r}{|z|} ds(z)- \int_{\partial \Omega(r)}
\log \frac{r}{|z|}  \partial_{\bf n} \bigl ( \log  H[\sigma] (z) \bigr) ds (z) =
$$
$$
= - \int_{\Omega(r)} \Delta \bigl ( \log  H[\sigma] (z) \bigr) \log \frac{r}{|z|} d A(z).
$$
For brevity, we record the above equation as ${\rm I}(r) - {\rm II}(r) = - {\rm III}(r)$.

\medskip

{\em Step 2.}
Since $H[\sigma] (z) = u(z) / (1-|z|^2)$, Harnack's inequality tells us that
$$
\log \frac{r}{|z|} \cdot \partial_{\bf n} \bigl ( \log H[\sigma] (z) \bigr)
$$
 is uniformly bounded  in $\Omega(r)$, independent of $1/2 <r<1$. Consequently, the second integral ${\rm II}(r)$ is bounded by a constant multiple of $\ell(Q)$ and
 ${\rm III}(r) = - {\rm I}(r) + O(\ell(Q))$.
  
 \medskip
 
 {\em Step 3.} We now focus on estimating the first integral ${\rm I}(r)$. Let $$E(r) = \bigl \{ \xi \in \partial  \mathbb{D} : r \xi \in \Omega(r) \bigr \}$$ and $\mathcal{A} (r) \subset \mathcal A(Q, M)$ be the subcollection of squares with $\ell (Q_j) > 1- r$. As $\partial_{\bf n} \log \frac{r}{|z|}$ vanishes on the radial portion of $\partial \Omega(r)$, ${\rm I}(r)$ is equal to
$$
\int_{\partial Q \cap \{ 1 - |z| = \ell(Q) \} } 
\log \bigl ( H[\sigma] (z) \bigr ) ds(z) - \sum_{Q_j  \in \mathcal{A}(r)} \int_{\partial Q_j \cap \{ 1 - |z| = \ell(Q_j) \} }  \log \bigl (  H[\sigma] (z)\bigr ) ds (z)
$$
$$
- \int_{E(r)}  \log \bigl ( H[\sigma] (r \xi) \bigr ) dm(\xi) + O(\ell(Q)),
$$
where the error term comes from replacing $ \partial_{\bf n} \log (r/|z|) = \mp 1/|z|$ with $\mp 1$ and using the estimate
$ H[\sigma] (z) \leq \sigma (\partial \mathbb{D})/(1-|z|)^2$ for $z \in \mathbb{D}$.

By Harnack's inequality, $\log  H[\sigma] (z)  -\log  H[\sigma] (w)  = O(1)$  if the hyperbolic distance $d_h(z, w) = O(1)$. Therefore,
$$
{\rm I}(r) \, = \, \log \bigl (H[\sigma] (z_Q) \bigr ) \ell(Q)
- \sum_{Q_j \in \mathcal{A}(r)} \log \bigl ( H[\sigma] (z_{Q_j}) \bigr ) \ell(Q_j)
$$
$$
- \int_{E(r)} \log \bigl ( H[\sigma] (r \xi) \bigr ) dm(\xi) + O(\ell(Q)).
$$
By construction,
$\log H[\sigma] (z_{Q_j}) = \log H[\sigma] (z_Q) + \log M + O(1)$ while
$$
\log H[\sigma] (z_Q) + \log M + O(1) \, \ge \, 
\log H[\sigma] (r \xi ) \, \ge \, \log H[\sigma] (z_Q) - O(1),
$$
for any $ \xi \in  E(r)$.
 As $\sum_{Q_j \in \mathcal{A}(r)} \ell (Q_j) + m(E) = \ell (Q)$, the above estimates imply that
 $$
- \log M \cdot \ell(Q) - O(\ell(Q)) \, \le \, {\rm I}(r) \, \le \, - \log M \cdot \sum_{Q_j \in \mathcal A(r)} \ell(Q_j) + O(\ell(Q))
 $$
 and so
\begin{equation}
\label{eq:IIIr-estimate}
 \log M \cdot \ell(Q) + O(\ell(Q)) \, \ge \, {\rm III}(r) \, \ge \, \log M \cdot \sum_{Q_j \in \mathcal A(r)} \ell(Q_j) - O(\ell(Q)).
\end{equation}
  
 {\em Step 4.}
Finally, since the function $\log H[\sigma]$ is subharmonic, the integrals ${\rm III}(r)$ increase to
$$
\int_{\Omega} \Delta \bigl ( \log  H[\sigma] (z) \bigr) \log \frac{1}{|z|} d A(z)
$$
as $r \to 1$. As the estimates in (\ref{eq:IIIr-estimate}) are uniform in $1/2 < r < 1$, the lemma follows after taking $r \to 1$.
\end{proof}

\section{Lyapunov exponents}

Let $F$ be an inner function with $F(0) = 0$ and $\{ \sigma_\alpha : \alpha \in \partial \mathbb{D}\}$ be the collection of its Aleksandrov-Clark measures. As explained in the introduction, each measure $\sigma_\alpha$ is a singular probability measure on the unit circle supported on the set of points where $F$ has radial limit $\alpha$. Recall that if the measure
$$
\sigma_\alpha = \sum_{F(\beta)=\alpha} |F'(\beta)|^{-1} \delta_\beta, \qquad \alpha \in \partial \mathbb{D}, 
$$
is discrete, then its Lyapunov exponent $\chi(\sigma_\alpha, F)$ is given by 
$$
\chi(\sigma_\alpha, F)  = \int_{\partial \mathbb{D}} \log |F'(\xi)| d \sigma_\alpha (\xi) = \sum_{F(\beta)=\alpha} |F'(\beta)|^{-1} \log |F'(\beta)|, \qquad \alpha \in \partial \mathbb{D}. 
$$

\begin{lemma}
\label{nevanlinna-ac}
An inner function $F$ has finite entropy if and only if for $m$ a.e.~$\alpha \in \partial \mathbb{D}$, the measure $\sigma_\alpha$ is discrete and the function 
 $\chi(\sigma_\alpha, F)$ belongs to $L^1(\partial \mathbb{D})$.
\end{lemma}

\begin{proof}
Suppose that $F$ has finite entropy. Since $|F' (\xi)|< \infty$, for $m$ a.e.~$\xi \in \partial \mathbb{D}$, the measure $\sigma_\alpha$ is discrete for $m$ a.e.~$\alpha \in \partial \mathbb{D}$. By the Aleksandrov disintegration theorem
$$
dm(\xi) = \int_{\partial \mathbb{D}} d\sigma_\alpha(\xi) dm(\alpha),
$$
we have
\begin{align}
\label{ADT}
\int_{\partial \mathbb{D}} \log |F'(\xi)| dm(\xi) & = \int_{\partial \mathbb{D}} \int_{\partial \mathbb{D}} \log |F'(\xi)| d\sigma_\alpha(\xi) dm(\alpha) \nonumber \\
& = \int_{\partial \mathbb{D}} \chi(\sigma_\alpha, F) dm(\alpha). 
\end{align}
Conversely, assume that $\sigma_\alpha$ is a discrete measure and $\chi(\sigma_\alpha, F) < \infty$ for $m$ a.e.~$\alpha \in \partial \mathbb{D}$. Let $E \subset \partial \mathbb{D}$ be the Lebesgue measure zero set of points $\alpha \in \partial \mathbb{D}$ for which $\sigma_\alpha$ is not discrete. By L\"owner's lemma, its pre-image $F^{-1}(E) \subset \mathbb{D}$ also has Lebesgue measure zero and so $|F'(\xi)|< \infty $ for $m$ a.e.~$\xi \in \partial \mathbb{D}$. 
Consequently, if $\chi(\sigma_\alpha, F) \in L^1(\partial \mathbb{D})$, then the identity \eqref{ADT} makes sense and shows that $F$ has finite entropy. 
\end{proof}

The following lemma says that the Lyapunov exponent $\chi(\sigma_\alpha, F)$ defines a lower semicontinuous function on the unit circle:

\begin{lemma}
\label{lsc-lyapunov}
Let $F$ be an inner function.
Suppose $\alpha_n$ is a sequence of points on the unit circle converging to $\alpha$. If the measures $\{\sigma_{\alpha_n}\}$ are discrete and the Lyapunov exponents $\chi(\sigma_{\alpha_n}, F)$ are uniformly bounded, then
$\sigma_{\alpha}$ is also discrete and $\chi(\sigma_\alpha, F) \le \liminf_{n \to \infty} \chi(\sigma_{\alpha_n}, F)$.
\end{lemma}

In view of Julia's lemma, for any $\alpha \in \partial \mathbb{D}$,
$$
|F'(\alpha)|= 
\limsup_{z \to \alpha} \frac{1-|z|^2}{|z-\alpha|^2} \cdot
\frac{|F(\alpha)-F(z)|^2}{1-|F(z)|^2}.
$$
In particular, if $\alpha_n \to \alpha$, then
$
|F' (\alpha)| \leq \liminf_{n \to \infty} |F'(\alpha_n)|
$
and so the function $\alpha \to |F'(\alpha)|$ is lower semicontinuous on the unit circle.

\begin{proof}
By composing $F$ with a M\"obius transformation, we may assume that $F(0) = 0$. There exists a constant $M>0$ such that $\chi(\sigma_{\alpha_n}, F) \le M$ for any $n \in \mathbb{N}$. From the definition of the Lyapunov exponent, it is easy to see that
$$
\sigma_{\alpha_n} \bigl (\{ \beta : |F'(\beta)| \le N \} \big ) \ge 1 - \frac{M}{\log N}, \qquad n = 1, 2,\dots.
$$
As $\sigma_{\alpha_n}$ converges weakly to $\sigma_\alpha$ and $\beta \to |F'(\beta)|$ is lower semicontinuous, we also have
$$
\sigma_{\alpha} \bigl (\{ \beta : |F'(\beta)| \le N \} \big ) \ge 1 - \frac{M}{\log N}.
$$
Consequently, the limiting measure $\sigma_\alpha$ is discrete and it makes sense to talk about its Lyapunov exponent. The weak convergence $\sigma_{\alpha_n} \to \sigma_\alpha$ and the lower semicontinuity of $\beta \to \log |F'(\beta)|$ give
$
\chi (\sigma_\alpha, F) \leq \liminf_{n \to \infty} \chi (\sigma_{\alpha_n}, F).
$
\end{proof}

In the remainder of this section, we show Theorem \ref{ac-measures-for-bmo-inner-functions} which says that for AHPA mappings, $\chi(\sigma_\alpha, F)$ coincides with a $C^\infty(\partial \mathbb{D})$ function on a set of full Lebesgue measure.

\subsection{Weighted composition operators and APHA maps}

Let $F$ be an inner function satisfying the conditions of Theorem \ref{main-thm} with $F(0) = 0$. As usual, we write $O_{F'}$ for the outer part of $F'$. We consider the family of weighted composition operators
\begin{equation}
\label{eq:wco}
 C_s g = (g \circ F) O_{F'}^{-s}
\end{equation}
acting on the weighted Bergman spaces $A^2_\gamma$ with $\gamma > -1$, which consist of analytic functions $g$ on $\mathbb{D}$ with
$$
\|g\|_{A^2_\gamma}^2 = \frac{1}{c_\gamma} \int_{\mathbb{D}} |g (z)|^2 (1-|z|^2)^{\gamma} dA(z) < \infty,
$$
where the constant $c_\gamma$ is chosen so that $\| {\bf 1} \|_{A^2_\gamma} = 1$ and ${\bf 1}$ is the function that is identically 1 on the unit disk.
In terms of Taylor series, we have
\begin{equation*}
g(z) = \sum_{n=0}^\infty a_n z^n \qquad \implies \qquad
\|g\|_{A^2_\gamma}^2 \asymp
\sum_{n=0}^\infty (n+1)^{-1-\gamma} |a_n|^2.
\end{equation*}

\begin{lemma}
\label{sigma-half-plane}
Fix a negative real number $\sigma < 0$. Then, for $\gamma \ge -1 -2 \sigma$, the weighted composition operators $C_s$ define a differentiable function from the interval $(\sigma, \infty)$ to the Banach space of bounded linear operators from $A^2_\gamma$ to itself. Furthermore, the Banach space derivative $\dot C_s = (d/ds) C_s$ coincides with the pointwise derivative $g(z) \to (d/ds)|_{s=0} (C_s g(z)).$
\end{lemma}

\begin{proof} 

{\em Step 0.} By the Littlewood-Paley identity, we have
$$
\|C_s g \|_{A^2_\gamma}^2 \asymp \bigl | g (0)^2 O_{F'}(0)^{-2s} \bigr | + \int_{\mathbb{D}} \bigl | \bigl ((g \circ F) (z)  O_{F'}^{-s} (z) \bigr )' \bigr |^2 (1-|z|)^{2+\gamma} dA(z).
$$
The constant term $g(0)^2 O_{F'}(0)^{-2s}$ varies differentiably in $s \in \mathbb{R}$ and its modulus is bounded by $|O_{F'}(0)|^{-2s} \| g \|_{A_\gamma^2}^2$. Meanwhile, the above integral is dominated by twice the sum ${\rm I} + {\rm II}$, where
$$
{\rm I} = \int_{\mathbb{D}} \bigl | (g \circ F)' (z) O_{F'}^{-s} (z) \bigr |^2 (1-|z|)^{2+\gamma} dA(z)
$$
and
$$
{\rm II} = s^2 \int_{\mathbb{D}} \bigl | (g \circ F (z)) \cdot O_{F'} (z)^{-s} (\log O_{F'} (z))' \bigr |^2 (1-|z|)^{2+\gamma} dA(z).
$$

\medskip

{\em Step 1.} By Statement (4) of Theorem \ref{main-thm}, we have
$$
|O_{F'} (z)| \asymp \frac{1-|F(z)|}{1-|z|}, \qquad  z \in \mathbb{D}. 
$$
A change of variables shows that
$$
{\rm I} \lesssim  \int_{\mathbb{D}} |g'(z)|^2 N(z) dA(z),
$$
where
$$
N(z) =  ( 1-|z|  )^{- 2 s}  \sum_{F(w) = z}   (1-|w|)^{2+\gamma + 2 s}.
$$
Note that by the remark following Lemma \ref{littlewood}, there exists a constant $C = C_F >0$ such that
$$
\sum_{F(w) = z}  (1-|w|) \leq C (1-|z|), \qquad z \in \mathbb{D}. 
$$
In particular, for any $\beta \geq 1$, we have
$$
\sum_{F(w) = z}   (1-|w|)^\beta \leq C^{\beta} (1-|z|)^{\beta}, \qquad z \in \mathbb{D}.
$$
Taking $\beta = 2+\gamma + 2 s \geq 1$, we deduce 
\begin{equation*}
 N(z)  = (1-|z|)^{- 2 s} \sum_{F(w) = z} (1-|w|)^{2+\gamma + 2 s} 
 \lesssim (1-|z|)^{2+\gamma}.
\end{equation*}
Consequently, ${\rm I} \lesssim \|g \|^2_{A^2_\gamma}$.
 
 \medskip

{\em Step 2.} We now estimate II under the additional assumption that $|O_{F'} (z)| \asymp |F'(z)|$ at all points $z \in \mathbb{D}$. While this simplifying assumption does not always hold, the computation below will serve as a guide in the general case. Since $\log O_{F'} $ lies in the Bloch space, we have
$$
{\rm II} \lesssim \int_{\mathbb{D}} \bigl | (g \circ F) (z) \cdot O_{F'} (z)^{-s} \bigr |^2 (1-|z|)^{\gamma} dA(z).
$$
Using that $|F'(z)| \asymp |O_{F'}(z)| \asymp (1-|F(z)|)/ (1-|z|)$, $z \in \mathbb{D}$, we get 
$$
{\rm II} \lesssim \int_{\mathbb{D}} \bigl | (g \circ F) (z) |^2  \biggr(\frac{1-|F(z)|}{1-|z|}\biggr)^{- 2 s - 2} |F'(z)|^2 (1-|z|)^{\gamma} dA(z).
$$
Now, a change of variables shows that
$$
{\rm II} \lesssim  \int_{\mathbb{D}} |g(z)|^2 N_2(z) dA(z),
$$
where
\begin{align*}
N_2(z) & = \sum_{F(w)=z}  \biggl ( \frac{1-|z|}{1-|w|} \biggr )^{- 2 - 2  s}  (1-|w|)^\gamma \\
& = (1-|z|)^{-2- 2  s}  \sum_{F(w)=z} (1-|w|)^{\gamma+2+2 s} . 
\end{align*}
Arguing as above shows that $N_2(z) \lesssim (1-|z|)^\gamma$ if $2+\gamma + 2 s \ge 1$. We conclude that ${\rm II} \lesssim \|g \|^2_{A^2_\gamma}$ holds under the assumption that $|O_{F'} (z)| \asymp |F'(z)|$ at all points $z \in \mathbb{D}$.

\medskip

{\em Step 3.} Recall that for an inner function $F$ satisfying the conditions of Theorem \ref{main-thm}, the inner part of $F'$ is a Carleson-Newman Blaschke product $B$. 
Corollary \ref{good-annuli} produces constants $m, \delta > 0$, $N \in \mathbb{N}$ and a collection of round annuli $A_c = A_h(c, r_c, R_c)$  associated to the critical points $\{c \in \mathbb{D} : F'(c)=0 \}$ of $F$ such that:
\begin{enumerate}
\item[(1)] $A_c \subset \{z \in \mathbb{D} : 1 < d_h (z,c) <2 \}$.
\item[(2)] $R_c/r_c \ge 1 + m$.
\item[(3)] $|B(z)| > \delta$ on any annulus $A_c$.
\item[(4)] The collection $\{A_c: F'(c) = 0 \}$ is {\em quasi-disjoint} in the sense that any point in the unit disk is contained in at most $N$ such annuli.
\end{enumerate}

By (2) and the maximum modulus principle or subharmonicity considerations, for any holomorphic function $\varphi$ on the unit disk we have
\begin{equation}
\label{annuli}
   \int_{B_{h}(c,1)} |\varphi(z)| (1-|z|)^{\gamma} dA(z) \lesssim \int_{A_c} |\varphi(z)| (1-|z|)^{\gamma} dA(z), \qquad c \in \mathbb{D}. 
\end{equation}
Form the sets
 $$U = \bigcup_{c:\,F'(c) = 0} B_h(c, 1) \qquad \text{and} \qquad \mathcal A = \bigcup_{c:\, F' (c)= 0} A_c.$$
By Lemma \ref{carleson-newman}(a), $|B(z)|$ is bounded below on $\mathbb{D} \setminus U$ by a positive constant that depends on $F$.

\medskip

{\em Step 4.} To give a rigorous proof for the estimate ${\rm II}  \lesssim \|g \|^2_{A^2_\gamma}$, we split the integral over $U$ and $\mathbb{D} \setminus U$. 
Since $|O_{F'} (z)| \asymp |F'(z)| $ for $z \in \mathbb{D} \setminus U$, the integral
$$
\int_{\mathbb{D} \setminus U} \bigl | (g \circ F ) (z) \cdot O_{F'} (z)^{-s} \bigr |^2 (1-|z|)^{\gamma} dA(z)
$$
may be handled as in Step 2.
The estimate \eqref{annuli} and the quasi-disjointness of the annuli $\{A_c: F'(c) = 0 \}$ yield
$$
\int_{U} \bigl| g ( F (z)) O_{F'} (z)^{-s} \bigr|^2 (1-|z|)^{\gamma} dA(z) \lesssim \int_{\mathcal A}  \bigl |  g  (F (z))  O_{F'} (z)^{-s} \bigr |^2 (1-|z|)^{\gamma} dA(z),
$$
which also fits into the framework of Step 2 as $|O_{F'} (z)| \asymp |F'(z)| $ for $z \in \mathcal A$.

\medskip

{\em Step 5.} With help of the dominated convergence theorem, it is not difficult to see that the operator norm of $C_t - C_s$ tends to $0$ as $t \to s$, which means that the operators $C_s$ vary continuously in $s \in (\sigma, \infty)$. 
To show that the operators $C_s$ vary differentiably in $s \in (\sigma, \infty)$ with Banach space derivative
$- \log O_{F'} \cdot C_s$, one needs to check that the operator norms of
$$\frac{C_t - C_s}{t-s} + \log O_{F'} \cdot C_s$$
tend to 0 as $t \to s$. This can be readily seen from the mean value theorem and the dominated convergence theorem.
We leave the details to the reader.
\end{proof}

\begin{proof}[Proof of Theorem \ref{ac-measures-for-bmo-inner-functions}]
By composing $F$ with a M\"obius transformation, we may assume that $F(0) = 0$.
 According to Lemma \ref{sigma-half-plane}, $\dot C_0$ is just the weighted composition operator
$$
g \to - (g \circ F) \log O_{F'}.
$$
Taking the dual with respect to the $L^2$ pairing on the unit circle, we see that the weighted transfer operator 
$$
\dot L_0 h(\alpha) = - \int_{\partial \mathbb{D}} \log(O_{F'}) h \, d\sigma_\alpha
$$
 acts on the Dirichlet-type spaces $\mathcal D_{1+\gamma}$ that are dual to $A^2_\gamma$, which consist of holomorphic functions $h(z) = \sum a_n z^n$ on the unit disk for which
$$
\sum_{n=0}^\infty (n+1)^{1+\gamma} |a_n|^2 < \infty.
$$
We refer the reader to \cite{inner-tdf} for an in-depth discussion on the duality between composition and transfer operators: the unweighted case is discussed in Sections 3.1 and 3.2, while the weighted case is presented in Section 4.1. 
Taking the constant function $h = {\bf 1}$, we see that 
$$
\re \dot L_0 {\bf 1}(\alpha) = - \chi(\sigma_\alpha, F)
$$
belongs to the Sobolev space $W^{\delta, 2}(\partial \mathbb{D})$ for any $\delta = \frac{1 + \gamma}{2} > 0$. Consequently, $ \chi(\sigma_\alpha, F)$ agrees with a $C^\infty(\partial \mathbb{D})$ function a.e.
\end{proof}

\subsection*{Acknowledgements}

This research was supported by the Israel Science Foundation (grant 3134/21), the Generalitat de Catalunya (grant 2021 SGR 00071), the Spanish Ministerio de Ciencia e Innovación (project PID2021-123151NB-I00) and the Spanish Research Agency (Mar\'ia de Maeztu Program CEX2020-001084-M).

\bibliographystyle{amsplain}

\begin{thebibliography}{00}

\bibitem[CMR06]{cima} J.~A.~Cima, A.~L.~Matheson, W.~T.~Ross, {\em The Cauchy Transform}\/, Mathematical Surveys and Monographs 125, American Mathematical Society, Providence, RI, 2006.

\bibitem[Cra91]{craizer} M.~Craizer, {\em Entropy of inner functions}\/, Israel J. Math. 74 (1991),  no. 2--3, 129--168.

\bibitem[Dya92]{dyakonov}  K.~M.~Dyakonov, {\em Smooth functions and co-invariant subspaces of the shift operator}\/, Algebra i Analiz 4 (1992), no. 5, 117--147, in Russian. English translation: St. Petersburg Math. J. 4 (1993), no. 5, 933--959.

\bibitem[Dya13]{dyakonov2} K.~M.~Dyakonov, {\em A Reverse Schwarz-Pick Inequality}\/, Computational Methods and Function Theory 13 (2013), no. 7--8, 449--457.

\bibitem[Gar07]{garnett} J.~B.~Garnett, {\em Bounded analytic functions}\/, Revised first edition, 
Graduate Texts in Mathematics 236, Springer, New York, 2007.

\bibitem[GM05]{garnett-marshall} J.~B.~Garnett, D.~E.~Marshall, {\em Harmonic  Measure}\/, New Mathematical Monographs 2, Cambridge University Press, 2005.

\bibitem[GP91]{GP} J.~B.~Garnett, M.~Papadimitriakis, {\em Almost isometric maps of the hyperbolic plane}\/, J.~London Math. Soc. 43 (1991), no. 2, 269--282.

\bibitem[Hei62]{heins} M.~Heins, {\em On a class of conformal metrics}\/, Nagoya Math. J. 21 (1962), 1--60.

\bibitem[Ivr19]{inner} O.~Ivrii, {\em Prescribing inner parts of derivatives of inner functions}, J. d'Analyse Math. 139 (2019), 495--519.

\bibitem[Ivr20]{stable} O.~Ivrii, {\em Stable convergence of inner functions}\/, J. London Math. Soc. 102 (2020), 257--286.

\bibitem[IK22]{IK22} O.~Ivrii, U.~Kreitner, {\em Critical values of inner functions}\/, preprint, 2022. arXiv:2212.14818.

\bibitem[IN22]{IN22} O.~Ivrii, A.~Nicolau, {\em Beurling-Carleson sets, Inner Functions and a Semilinear Equation}\/, preprint, 2022. arXiv:2210.01270. To appear in Analysis and PDE.

\bibitem[IU23]{inner-tdf} O.~Ivrii, M.~Urba\'nski, {\em Inner Functions, Composition Operators, Symbolic Dynamics and Thermodynamic Formalism}\/, preprint, 2023. arXiv:2308.16063.

\bibitem[IU24]{laminations} O.~Ivrii, M.~Urba\'nski, {\em Inner functions and laminations}\/, preprint, 2024. arXiv:2405.02878.

\bibitem[KL69]{KL} A.~Kerr-Lawson, {\em Some lemmas on interpolating Blaschke products and a correction}\/, Canadian J. Math. 21 (1969), 531–534.

\bibitem[Kra13]{kraus} D.~Kraus, {\em Critical sets of bounded analytic functions, zero sets of Bergman spaces
and nonpositive curvature}\/, Proc. London Math. Soc. 106 (2013), no. 4, 931--956.

\bibitem[KR12]{KR-survey} D.~Kraus, O.~Roth, {\em Critical Points, the Gauss Curvature Equation and Blaschke Products}\/, In: Blaschke Products and Their Applications, Fields Institute Communications 65 (2012), 133--157.

\bibitem[KR13]{KR-maximal} D.~Kraus, O.~Roth, {\em Maximal Blaschke products}\/, Adv. Math. 241 (2013), 58--78.

\bibitem[KRR07]{KRR} D.~Kraus, O.~Roth, S.~Ruscheweyh, {\em A boundary version of Ahlfors’ Lemma, locally complete conformal metrics and conformally invariant reflection principles for analytic maps}\/, J. d'Analyse Math. (2007) 101, 219--256.

\bibitem[Mas12]{mashreghi}J.~Mashreghi, {\em Derivatives of Inner Functions}\/, Fields Institute Monographs, 2012.

\bibitem[McM09]{mcmullen-rtree} C.~T.~McMullen, {\em Ribbon $\mathbb{R}$-trees and holomorphic dynamics on the unit disk}\/, J. Topol. 2 (2009), 23--76.

\bibitem[MN04]{MN} R.~Mortini, A.~Nicolau, {\em Frostman shifts of inner functions}\/, J. d'Analyse Math. 92 (2004), 285--326.

\bibitem[PS06]{PolSar} A.~Poltoratski, D.~Sarason, {\em Aleksandrov-Clark measures}\/, In: A.~L.~Matheson, M.~I.~Stessin, R.~M.~Timoney (eds), Recent advances in operator-related function theory, Contemp. Math. 393, Amer. Math. Soc., Providence, RI, 2006, 1--14.

\bibitem[Sak07]{Saks} E.~Saksman, {\em An elementary introduction to Clark measures}\/, 
 In: D.~Girela \'Alvarez, C.~González Enríquez (Eds.), Topics in complex analysis and operator theory, Univ. M\'alaga, M\'alaga, 2007, 85--136.

\bibitem[Sha93]{shapiro} J.~H.~Shapiro, {\em Composition Operators and Classical Function Theory}, Springer, New York, 1993.


\end{thebibliography}

\end{document}